\documentclass[preprint,12pt,reqno]{elsarticle}
\usepackage[utf8]{inputenc}
\usepackage[T1]{fontenc}
\usepackage[vmargin=30mm]{geometry}
\usepackage{amsmath,amssymb,amsfonts,amsthm,dsfont}
\usepackage{graphicx,verbatim,tikz,wrapfig}
\usepackage{physics}
\usepackage{subcaption}
\usepackage{calc,adjustbox,enumerate}
\usetikzlibrary{calc}

\usepackage{hyperref}
\numberwithin{equation}{section}
\numberwithin{figure}{section}
\theoremstyle{plain}
\newtheorem{theorem}{Theorem}[section]
\newtheorem{lemma}[theorem]{Lemma}
\newtheorem{proposition}[theorem]{Proposition}
\newtheorem{corollary}[theorem]{Corollary}
\theoremstyle{definition}
\newtheorem{definition}[theorem]{Definition}
\newtheorem{remark}[theorem]{Remark}

\newtheorem{example}[theorem]{Example}
\newtheoremstyle{ourexample}
  {1ex plus 1ex}   
  {\topsep}   
  {\normalfont}  
  {0pt}       
  {\bfseries} 
  {.}         
  {5pt plus 1pt minus 1pt} 
  {#1 \thmnote{E#3}} 
\theoremstyle{ourexample}
\newtheorem{ourexample}{Example}
\journal{Linear Algebra and its Applications}
\newcommand{\C}{\mathbb{C}} 
\renewcommand{\P}{\mathbb{P}} 
\newcommand{\R}{\mathbb{R}} 
\newcommand{\ii}{\mathrm{i}} 
\newcommand{\M}{\mathrm{M}}  
\newcommand{\cA}{\mathcal{A}} 
\newcommand{\cC}{\mathcal{C}} 
\newcommand{\cD}{\mathcal{D}} 
\newcommand{\cH}{\mathcal{H}} 
\newcommand{\cP}{\mathcal{P}} 
\newcommand{\cU}{\mathcal{U}} 
\newcommand{\cV}{\mathcal{V}} 
\DeclareMathOperator\ext{ext}     
\DeclareMathOperator\diag{diag}   
\DeclareMathOperator\conv{conv}   
\DeclareMathOperator\range{range} 
\DeclareMathOperator\rk{rk}       
\DeclareMathOperator\spn{span}    
\newcommand{\id}{\mathds{1}} 
\newcommand{\scp}[1]{\left\langle#1\right\rangle}

\newcommand{\tp}{{}^\top} 
\newcommand{\JNR}{\tilde{W}}
\title{Joint numerical ranges of three hermitian $4\times 4$ matrices}
\newlength{\strutheight}
\begin{document}

\author[1]{Piotr Pikul}
\affiliation[1]{organization={Jagiellonian University}, city={Krak\'{o}w},
country=Poland}
\ead{piotr.pikul@im.uj.edu.pl}
\author[2]{Ilya Spitkovsky}
\affiliation[2]{organization={New York University Abu Dhabi (NYUAD)}, country={United Arab Emirates}}
\ead{ims2@nyu.edu}
\author[3]{Konrad Szyma{\'n}ski}
\affiliation[3]{organization={Research Center for Quantum Information, Slovensk\'{a} Akad\'{e}mia Vied}, city=Bratislava, country=Slovakia}
\ead{k.sz@quantumstat.es}
\author[4]{Stephan Weis}
\affiliation[4]{organization={Faculty of Electrical Engineering, Czech Technical University in Prague}, country={Czech Republic}}
\ead{stephan.weis@fel.cvut.cz}
\author[1,5]{Karol~{\.Z}yczkowski}
\affiliation[5]{organization={Center for Theoretical Physics,
Polish Academy of Sciences}, city=Warsaw, country=Poland}
\ead{karol@cft.edu.pl}


\begin{abstract}
We analyze the joint numerical range $W$ of three hermitian matrices of order four.
In the generic case, this three-dimensional convex set has a smooth boundary.
We analyze non-generic structures. Fifteen possible classes regarding the numbers of 
non-elliptic faces in the boundary of $W$ are identified and an explicit example is 
presented for each class. Secondly, it is shown that a nonempty intersection of 
three mutually distinct one-dimensional faces is a corner point. Thirdly, introducing 
a tensor product structure into $\C^4=\C^2\otimes\C^2$, one defines the separable 
joint numerical range --- a subset of $W$ useful in studies of quantum entanglement. 
The boundary of the separable numerical range is compared with that of $W$.
\end{abstract}

\maketitle

%
\section{Introduction}
The \emph{joint numerical range} \cite[p.~23]{BonsallDuncan1971} of $A_1,\dots,A_k\in\M_n$, 
with respect to the algebra $\M_n$ of $n\times n$ matrices with complex entries, is the 
compact convex set 
\begin{equation}\label{eq:JNR-BonsallDuncan}\textstyle
W:=W(A_1,\dots, A_k)
:=
\big\{ (\tr \rho A_1, \ldots, \tr \rho A_k)\tp 
\colon 
\rho\in\M_n,\, \rho\geq 0,\, \tr\rho=1 \big\}.
\end{equation}
Here, $\tr X$ is the trace and $X\geq 0$ means  $X\in\M_n$ is positive semidefinite.
Equation \eqref{eq:JNR-BonsallDuncan} is a somewhat less common definition of 
the joint numerical range than
\begin{equation}\label{eq:JNR-standard}
\JNR 
:=\JNR (A_1,\dots,A_k)
:=\{\braket{\varphi}{A_i\varphi}_{i=1}^k
\colon
\ket{\varphi}\in\C^n, \braket{\varphi}=1\},
\end{equation}
where $\braket{\varphi_1}{\varphi_2}:=\overline{x_1}y_1+\cdots+\overline{x_n}y_n$ is the 
inner product of $\ket{\varphi_1}=(x_1,\ldots,x_n)\tp$ and $\ket{\varphi_2}=(y_1,\ldots,y_n)\tp$ 
in $\C^n$. In any case, $W$ is the convex hull of $\JNR $. 
\par
For a single matrix $B\in\M_n$, the set $\JNR (B)$ is also called \emph{numerical range} 
\cite{Toeplitz1918,Hausdorff1919}. Identifying $\C=\R^2$, we have
\begin{equation}\label{eq:NR}\textstyle
\JNR (B)=\JNR (\Re B,\Im B),
\end{equation}
as $B=\Re B+\ii\Im B$, where $\Re B:=\frac{1}{2}(B+B^\ast)$, 
$\Im B:=\frac{1}{2\ii}(B-B^\ast)$, and $B^\ast$ is the hermitian conjugate of $B$.
Similarly, $W,\JNR \subset\C^k$ can be identified with subsets of $\R^{2k}$. Hence, 
unless stated otherwise, we assume $A_1,\dots,A_k$ are hermitian. So, $W,\JNR \subset\R^k$. 
It is well known that $\JNR $ is convex if $k\leq 2$ \cite{Toeplitz1918,Hausdorff1919} or 
if $k=3$ and $n\geq 3$ \cite{Au-YeungPoon1979,MaierNetzer2024}. In particular, the numerical 
range is $W(B)=\JNR (B)$, and $W(A_1,A_2,A_3)=\JNR (A_1,A_2,A_3)$ holds for hermitian
$4\times 4$ matrices $A_1,A_2,A_3$.
\par
The numerical range of an operator is a well-known notion studied in operator theory and 
linear algebra 
\cite{Toeplitz1918,Hausdorff1919,Kippenhahn1951,Fiedler1981,GustafsonRao1997,Gutkin2004,HornJohnson2012}, 
while its generalizations find several applications in theoretical  physics 
\cite{Schulte-Herbrueggen-etal2008,KPL+09,DGH+11,Jevtic2014}.
The joint numerical range is well known in matrix theory
\cite{AdamMaroulas2002,Au-YeungPoon1979,BindingLi1991,BonsallDuncan1971,ChienNakazato2010,
Gutkin-etal2004,LPW20,MuellerTomilov2020,Plaumann-etal2021,Szymanski-etal2018,Weis2011} 
and theoretical physics \cite{GZ13,Xie2020,SCSZ21}.
It is also known as the set of states \cite{Paulsen2003} of a matrix system 
\cite{ChoiEffros1977,Arveson2010}, an algebraic object with numerous applications 
in quantum information theory, see for example 
\cite{Heinosaari-etal2013,Paulsen-etal2016}.
\par
The subset $W$ of $\R^k$ has a clear interpretation in quantum physics:
It is the set of possible expectation values of joint
measurements of $k$ observables $A_1,\dots, A_k$. The case of $k=3$ hermitian matrices of 
size $n=3$ was studied in~\cite{Szymanski-etal2018}, in which all possible shapes of 
$W$ were divided into $10$ classes with respect to the structure of its boundary.
Here, we shall go beyond this case and analyze the problem of three hermitian matrices of size 
$n=4$. This step forward has important consequences for quantum theory: As four is a 
composite number, one can introduce a tensor product structure, $\C^4=\C^2\otimes\C^2$,
which corresponds to a two-qubit system. Such a splitting of a four-dimensional system
into two subsystems allows for the introduction of \emph{product states}, which convey no 
correlations between the two parts, and \emph{separable states}, which encode classical 
correlations -- all the other states are \emph{entangled}, and among them are those whose  correlations can never 
be explained by classical probability theory \cite{Werner1989}. By restricting the range of states in the 
definition of numerical range one arrives at product \cite{PGM+11} and separable numerical 
range \cite{SCSZ21}, which physically correspond to classically achievable sets of 
expectation values. The geometry of bipartite quantum systems is also studied outside the 
algebraic setting of numerical ranges, see for instance 
\cite{HanKye2025,Liss-etal2024,Morelli-etal2024,Sanpera-etal1998}.
\par

This work investigates the joint numerical range $W=W(A_1,A_2,A_3)$ of three 
hermitian matrices $A_1,A_2,A_3$ of size $n=4$. Since $k\leq 3$, the von Neumann-Wigner 
non-crossing theorem implies that $W$ is generically strictly convex and has a smooth
boundary $\partial W$\!, i.e.\ the set of triples for which $W$ is strictly convex and 
has a smooth boundary is open and dense in the set of all triples of (complex) 
hermitian $4\times 4$ matrices with the Euclidean topology, see 
\cite[Proposition~4.9]{Gutkin-etal2004} and \cite[Theorem~4.2]{Szymanski-etal2018}. 
In the analogous sense of being true for triples of real symmetric $4\times 4$ matrices 
in an open and dense set, $W$ has generically an even number of at most ten elliptic 
disks on the boundary and it has no other flat portions on $\partial W$. This follows 
from \cite[Theorem~1.1]{Ottem-etal2015}, see also \cite{HelsoRanestad2021}, by duality 
as described in Remark~\ref{rem:Ottem} below.
\par
In this paper we focus on non-elliptic flat portions on $\partial W$. Any flat portion 
on $\partial W$ corresponds to the numerical range of a $3\times 3$ matrix $B$, see 
Remark~\ref{rem:faces-JNR}~(c),~(d) for details. The classification of these numerical ranges 
was achieved by Kippenhahn \cite{Kippenhahn1951}. See also \cite{Keeler-etal1997,GauWu2021} 
and, for more rigorous proofs, see \cite{ChienNakazato2010,Plaumann-etal2021}. Let
\begin{equation}\label{eq:hyperbolic}
p_B(u_0,u_1,u_2)
:=\det(u_0\id_3+u_1\Re B+u_2\Im B),
\qquad
u_0,u_1,u_2\in\C,
\end{equation}
and let $C(B)$ be the dual curve \cite{Fischer2001} to the curve $p_B=0$ that 
consists of all tangent lines to $p_B=0$. Identifying the complex projective plane 
$\C\P^2$ with its dual space of lines, we define
$C_\R(B):=\{(x_1,x_2)\tp\in\R^2\mid (1:x_1:x_2)\in C(B)\}$. The set $C_\R(B)$ is known 
as the \emph{Kippenhahn curve} \cite{Bebiano-etal2021,GauWu2021} and $\JNR (B)$ is the 
convex hull of $C_\R(B)$. This result divides the shape of $\JNR (B)$ into four different 
classes \cite{GauWu2021}:
\begin{enumerate}
\item 
$C_\R(B)$ consists of three points that are the eigenvalues of $B$, and $\JNR (B)$ is the 
closed triangular region whose vertices are these points.
\item
$C_\R(B)$ consists of an ellipse and a point. Depending on whether the point lies 
inside the ellipse or not, $\JNR (B)$ is an elliptic disk or it has a \emph{droplet shape}.
\item
$C(B)$ is an irreducible curve of degree four that has a double tangent,
$C_\R(B)$ is of heart shape, and the boundary of $\JNR (B)$ contains
a line segment but no corner.
\item
$C(B)$ is an irreducible curve of degree six, $C_\R(B)$ consists of two parts, 
one inside the other. The outer part is an oval (a smooth and strictly convex curve), 
whose convex hull is $\JNR (B)$.
\end{enumerate}
For simplicity (see Lemma~\ref{lem:int2non-elliptic} and Example~\ref{ex:2rank3ellipses})
we focus on flat parts of the boundary $\partial W$
different from elliptic disks. This prompts the following definitions.
\par
A \emph{face} of a convex subset $C$ of $\R^d$ is a convex subset of $C$ that contains the 
endpoints $x,y\in C$ of every open segment $\{(1-\lambda)x+\lambda y:\lambda\in(0,1)\}$ it 
intersects.
An \emph{exposed face} of $C$ is the set of minimizers of any linear functional 
on $C$. For technical reasons, $\emptyset$ is also considered an exposed face. It is well 
known that every exposed face of $C$ is a face of $C$. Faces that are no exposed faces are 
called \emph{nonexposed faces} \cite{Rockafellar1970}.
\par
\begin{definition}[Non-elliptic faces]\label{def:non-elliptic}
Let $A_1,A_2,A_3$ be hermitian $4\times 4$ matrices such that $W=W(A_1,A_2,A_3)$ has 
dimension $\dim(W)=3$. A \emph{non-elliptic face} of $W$ is any two-dimensional face 
of $W$ that is not an elliptic disc. We say a non-elliptic face of $W$ has \emph{type} 
$m\in\{0,1,2,3\}$ if it has $m$ one-dimensional faces.
\end{definition}
\begin{table}[htb]
\centering
\begin{tabular}{c|l|l}
type & shape & Kippenhahn curve \\\hline\hline 
$0$ & oval & degree-6 curve \\\hline
$1$ & loaf & degree-4 curve \\\hline
$2$ & droplet & ellipse and a point outside the ellipse \\\hline 
$3$ & triangle & three affinely independent points \\\hline\hline 
\end{tabular}
\caption{The four types and shapes of non-elliptic faces and their 
Kippenhahn curves.}
\label{tab:types}
\end{table}
\begin{figure}[htb]
\includegraphics[width=.32\textwidth]{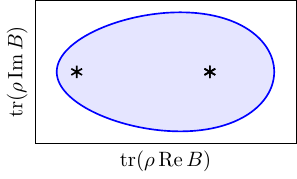}
\includegraphics[width=.32\textwidth]{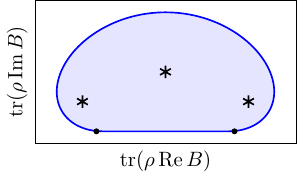}
\includegraphics[width=.32\textwidth]{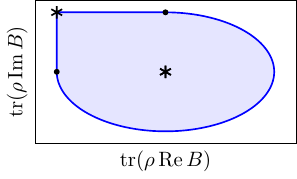}
\caption{Examples of numerical ranges of a $3\times 3$ matrix $B$ in the complex plane
exemplifying the shapes of type 0 (oval), 1 (loaf) and 2 (droplet). 
Black asterisks mark the (in some cases, duplicate) eigenvalues\protect\footnotemark{}
of $B$.
Small black dots mark the nonexposed faces of the numerical ranges --- endpoints of segments 
smoothly connected to the rest of the boundary. The missing numerical range of type 3 is a 
triangle, which is the convex hull of the three eigenvalues of a normal matrix, $BB^*=B^*B$.} 
\label{fig:num3x3}
\end{figure}
\footnotetext{Note that the image scalings are not proportional (aspect ratios differ 
from one). Having a double eigenvalue at its coinciding foci, the ellipse in the third 
drawing depicts a circle.}
Non-elliptic faces being numerical ranges of $3\times 3$ matrices, each type 0, 1, 2, 
and 3 refers to a unique shape in Kippenhahn's classification above, see Table~\ref{tab:types}.
Exemplary matrices for each type read,
\[
\begin{bmatrix}
 -1 & 1 & 1 \\
 0 & \frac 12 & 1 \\
0 & 0 & \frac12
\end{bmatrix},
    \qquad
\frac12 \begin{bmatrix}
1 & 0 & -\ii \\
0 & -1 & -\ii \\
\ii & \ii & 2\ii
\end{bmatrix},
    \qquad
\begin{bmatrix}
1 & 2 & 0 \\
0 & 1 & 0 \\
0 & 0 & \ii
\end{bmatrix},
    \qquad
\begin{bmatrix}
-1 & 0 & 0 \\
0 & 1 & 0 \\
0 & 0 & \ii
\end{bmatrix}.
\]
The numerical ranges for the first three of the above matrices of type $0,1$ and $2$, are depicted
in Figure~\ref{fig:num3x3}. 
The numerical range of the diagonal matrix of type 
$3$ is the convex hull of the eigenvalues: The triangle with vertices at $-1,1$ and $\ii$.
\par
We continue the paper by recalling and proving preliminaries on the algebra and 
geometry of the joint numerical range in Section~\ref{sec:recap}.
Section~\ref{sec:non-elliptic} contains one of the main 
results of the paper, Theorem~\ref{thm:non-elliptic}, which establishes
15 classes that capture all possible numbers of non-elliptic face of
the joint numerical range $W$ of three hermitian $4\times 4$ matrices. 
Section~\ref{sec:examples} presents an example of $W$ for every class. A second 
major result is Theorem~\ref{thm:3segments} of Section~\ref{sec:onedim}, according 
to which a non-empty intersection of three mutually distinct one-dimensional faces 
of $W$ is a corner point of $W$. Section~\ref{sec:elliptic-faces} collects 
observations and examples regarding ellipses on the boundary of $W$ that motivated 
us to initiate this work. In Section~\ref{sec:restricted-JNRs}, we compare the 
structure of the standard joint numerical range $W$ with the separable numerical 
range.
\par
%
%
\section{Algebra and geometry of the joint numerical range}
\label{sec:recap}
Summarizing prior work of one of the authors \cite{Weis2011,Weis2012}, this section 
presents an algebraic description of the faces of the joint numerical range. 
Normal cones are reviewed, too, and properties of corner points are discussed, 
including the assertion proved by Binding and Li \cite{BindingLi1991} that 
every corner point of the joint numerical range is contained in the joint 
spectrum. Bonsall and Duncan's notion of the joint numerical range with respect to 
a *-algebra \cite{BonsallDuncan1971} is convenient for the analysis of faces.
\par
The algebra $\M_n$ of $n\times n$ complex matrices is a unitary space with the 
Frobenius inner product $\scp{A,B}:=\tr(A^\ast B)$, $A,B\in\M_n$. It is a 
*-algebra with the involution $A\mapsto A^\ast$ mapping a matrix $A\in\M_n$ to 
its conjugate transpose. The algebra $\M_n$ is partially ordered by $A\leq B$,
or equivalently $B\geq A$, meaning that $B-A$ is positive semidefinite. Let 
$\id_n$ be the $n\times n$ identity matrix. We define a \emph{*-algebra} on $\C^n$ 
to be a complex linear subspace of $\M_n$ that is closed under matrix multiplication 
and under the involution $A\mapsto A^\ast$.
\par
Let $\cA$ be a *-algebra on $\C^n$, $\cH_\cA:=\{A\in\cA\colon A^\ast=A\}$ be the
real vector space of hermitian matrices, and let $A_1,\ldots,A_k\in\cH_\cA$. Let also
\begin{equation}\label{eq:measuremap}
w=w_{A_1,\dots,A_k}:\cH_\cA\to\R^k,
\qquad
w(A):=(\scp{A,A_1},\dots,\scp{A,A_k})\tp.
\end{equation}
The image of the set of \emph{density matrices}
\begin{equation}
\cD_\cA:=\{\rho\in\cA\colon 
\text{$\rho\geq 0$ and $\tr(\rho)=1$} \}
\label{eq:setofdensmat}
\end{equation}
under $w$ is the \emph{joint numerical range}\footnote{%
To be precise, the joint numerical range is defined in \cite{BonsallDuncan1971} 
in terms of states on the algebra $\cA$, that is to say, normalized positive 
linear functionals. This is equivalent to the aforementioned definition because 
states are represented by density matrices, see the Sections~2.3.2 and~2.4.3 in 
\cite{BratteliRobinson1987}.}
of $A_1,\ldots,A_k$ with respect to the *-algebra $\cA$, which we denote by
\begin{equation}\label{eq:JNR-A}
W_\cA(A_1,\ldots,A_k):=w(\cD_\cA),
\end{equation}
see \cite[Definition~11]{BonsallDuncan1971}. We are mostly interested in the joint numerical 
range 
\begin{equation}\label{eq:JNR}
W_{\M_n}(A_1,\ldots,A_k) 
= 
\{ (\tr \rho A_1, \ldots, \tr \rho A_k)\tp 
\colon 
\rho\in\M_n, \rho\geq 0, \tr\rho=1 \},
\end{equation}
with respect to the full matrix algebra $\M_n$, which we also denote by $W(A_1,\ldots,A_k)$ or 
simply by $W$. The set $W$ is also known as the ``algebraic numerical range'' 
\cite{MuellerTomilov2020}.
\par
It is well known \cite{AlfsenShultz2001} that there is an isomorphism 
between the set of exposed faces of $\cD_\cA$ and the set
$\cP_\cA:=\{P\in\cA\colon P=P^2=P^\ast\}$ of orthogonal projections.
Namely, if $A\in\cA$ is a hermitian matrix, 
then the set of minimizers of the real affine functional $\cD_\cA\to\R$, 
$\rho\mapsto\scp{\rho,A}$ is the set $\cD_{P\cA P}$ of density matrices of 
the *-algebra $P\cA P$, where $P$ is the spectral projection of $A$ in the 
*-algebra $\cA$ corresponding to the smallest spectral value\footnote{%
The use of spectral values (as opposed to eigenvalues) is mandatory, 
because the identity $Q$ of the *-algebra $Q\M_nQ$ differs from $\id_n$ for 
every $n\times n$ orthogonal projection $Q\neq\id_n$. Spectral values were 
introduced into this finite-dimensional setting in the arXiv-version of 
\cite{Weis2011}, see also the erratum to \cite{Weis2011}.}
of $A$ with respect to $\cA$. The map $P\mapsto\cD_{P\cA P}$ is a lattice 
isomorphism from $\cP_\cA$ to the set of exposed faces of $\cD_\cA$. Thereby, 
$\cP_\cA$ is partially ordered by $P\leq Q$ (the range of $P$ is a subset of 
the range of $Q$). The set $\cP_\cA$ is a complete lattice 
in this partial order, which means any subset of $\cP_\cA$ has an infimum and 
a supremum. The set of exposed faces is partially ordered by inclusion and 
the infimum is the intersection. 
\par
A similar isomorphism exists between the faces of the joint numerical range and certain 
orthogonal projections \cite{Weis2011}. However, whereas all faces of $\cD_\cA$ are 
exposed faces, the joint numerical range $W$ can have nonexposed faces. Fig.~\ref{fig:num3x3} 
shows examples: the endpoints of flat portions on the boundary of the numerical range
that are joined by curved boundary portions. Two ideas suffice to describe the 
faces of the joint numerical range. First, as with exposed faces, the set of 
faces of a convex subset $C$ of a Euclidean space, partially ordered by inclusion, 
is a complete lattice, whose infimum is the intersection 
\cite{LoewyTam1986,Weis2012}. Second, we define an \emph{access sequence}\footnote{%
An element of an access sequence with respect to a convex set has been called 
a \emph{poonem} in convex geometry \cite{Gruenbaum2003}. The term \emph{access sequence} 
stems from mathematical statistics \cite{CsiszarMatus2003}.}
\cite{Gruenbaum2003,CsiszarMatus2003}
of faces (of length $\ell$) with respect to $C$ to be a nonincreasing sequence
\begin{equation}\label{eq:access-F}
F_0 \supset F_1 \supset \dots \supset F_\ell
\end{equation}
subject to the conditions that $F_0=C$, that $F_{i+1}$ is an exposed face of 
$F_i$, $i=0,\ldots,\ell-1$, and that $F_\ell\neq\emptyset$. Clearly
\cite[Section~1.2.1]{Weis2012}, a nonempty subset $F$ of $C$ is a face of $C$ if 
and only if $F$ is an element of an access sequence with respect to $C$. This 
follows from a separation argument \cite[Lemma~4.6]{Weis2012} proving that 
every face of $C$ that is strictly included in $C$ is included in an exposed 
face of $C$ that is itself strictly included in $C$.
\par
The set of hermitian matrices, density matrices, and orthogonal projections in $\M_n$ 
is denoted by $\cH_n:=\cH_{\M_n}$, $\cD_n:=\cD_{\M_n}$ and $\cP_n:=\cP_{\M_n}$, 
respectively. Let $A_1,\ldots,A_k\in\cH_n$ denote arbitrary hermitian $n\times n$ 
matrices. The *-algebra 
\[
P\M_nP=\{PAP\colon A\in\M_n\}
\]
is defined in terms of a nonzero orthogonal projection $P\in\cP_n$. An 
\emph{access sequence} of projections (of length $\ell$) with 
respect to $A_1,\dots,A_k$ is a nonincreasing sequence 
\begin{equation}\label{eq:access-P}
P_0\geq P_1\geq\dots\geq P_\ell
\end{equation}
of orthogonal projections in $\cP_n$, subject to the conditions that $P_0=\id_n$ 
and that for every $i=0,\dots,\ell-1$ there exists a tuple of real numbers 
$s_1,\dots,s_k\in\R$ such that $P_{i+1}$ is the spectral projection of 
$P_i(s_1A_1+\dots+s_kA_k)P_i$ in the *-algebra $P_i\M_nP_i$ corresponding 
to the smallest spectral value of $P_i(s_1A_1+\dots+s_kA_k)P_i$ with respect to 
$P_i\M_nP_i$.
\par
It is easy to see that $W$ is affinely isomorphic to an orthogonal projection of 
$\cD_n$. Let $\pi_\cU:\cH_n\to\cH_n$ denote the orthogonal projection onto a 
subspace $\cU$ of $\cH_n$. If $\cU$ is the real linear span of $A_1,\dots,A_k$, then 
$w_{A_1,\dots,A_k}:\cH_n\to\R^k$ restricts to the affinity 
\begin{equation}\label{eq:aff-pi-W}
w_{A_1,\dots,A_k}:\pi_\cU(\cD_n)\to W
\end{equation}
(see \cite[Remark~1.1]{Weis2011} for details). Similarly,
$w_{\id_n,A_1,\dots,A_k}:\cH_n\to\R^{k+1}$ restricts to the affinity 
\begin{equation}\label{eq:aff-pi-1W}
w_{\id_n,A_1,\dots,A_k}:\pi_\cV(\cD_n)\to\{1\}\times W
\end{equation}
if $\cV=\spn\{\id_n,A_1,\dots,A_k\}$, because $\tr\rho=1$ holds for every $\rho\in\cD_n$. 
The following theorem follows from \cite[Theorem~3.7]{Weis2011} 
through a family of affinities of the type \eqref{eq:aff-pi-W}. More details can be found 
in \cite[Section~4.1]{Weis2014}, in particular in Lemma~4.19 therein.
\par
\begin{theorem}\label{thm:PF-iso}
There is a one-to-one correspondence between access sequences of orthogonal 
projections $P_0\geq P_1\geq\dots\geq P_\ell$ with respect to $A_1,\dots,A_k$ 
and access sequences of faces $F_0 \supset F_1 \supset \dots \supset F_\ell$
with respect to $W$. For every $i=0,\ldots,\ell$ we have $F_i=w(\cD_{P_i\M_nP_i})$
and $P_i\in\cP_n$ is the unique orthogonal projection such that 
$\cD_{P_i\M_nP_i}=w|_{\cD_n}^{-1}(F_i)$.
\end{theorem}

By Theorem~\ref{thm:PF-iso}, for every nonempty face $F$ of $W$
there is a unique orthogonal projection $P\in\cP_n$ that satisfies $\cD_{P\M_nP}=w|_{\cD_n}^{-1}(F)$.
We call $P$ the \emph{projection associated to} $F$
and refer to $F$ as a \emph{rank-$r$ face} of $W$ if $r=\rk(P)$.
\par
\begin{example}\label{exa:rankF}
Let $k=3$, let $W=W(A_1,A_2,A_3)$ be the joint numerical range of three hermitian 
$4\times 4$ matrices $A_1,A_2,A_3$, and assume $\dim(W)=3$. Let $F_1$ be a 
non-elliptic face of $W$, $F_2$ a one-dimensional face of $F_1$, and $p$ an 
endpoint of the segment $F_2$. After applying to $A_1,A_2,A_3$ the similarity 
$A\mapsto UAU^\ast$ with respect to a suitable unitary $U\in\M_4$, the access 
sequence of projections $\id_4\geq P_1\geq P_2 \geq P_3$ that corresponds to the 
access sequence of faces $W \supset F_1 \supset F_2 \supset \{p\}$ by 
Theorem~\ref{thm:PF-iso} is $P_1=\diag(1,1,1,0)$, $P_2=\diag(1,1,0,0)$, and 
$P_3=\diag(1,0,0,0)$.
\end{example}
We show in Remark~\ref{rem:faces-JNR} that every rank-$r$ face $F$ of $W$ is the 
joint numerical range of $k$ hermitian $r\times r$ matrices and that $F$ is 
isometric to the joint numerical range of $d=\dim(F)$ hermitian $r\times r$ 
matrices. To this end, we employ the leading principal submatrix $A[r]$ of a matrix 
$A\in\M_n$ obtained by deleting the rows and columns $r+1$ through $n$.
\par
Note that $F=w(\cD_{P\M_nP})$ holds for every face $F$ of $W$ and the projection 
$P$ associated to $F$ by Theorem~\ref{thm:PF-iso}. In general, for arbitrary nonzero 
$P\in\cP_n$, the set $w(\cD_{P\M_nP})$ is not a face of $W$ but it is still useful 
in Lemma~\ref{lem:convWs} and Corollary~\ref{cor:convWcorner}. 
\par
\begin{remark}\label{rem:faces-JNR}
Let $P\in\cP_n$ be a projection of rank $r>0$, and let 
$F:=w(\cD_{P\M_nP})$. Then:
\begin{enumerate}[\ (a)]
\item
As per the definition in equation \eqref{eq:JNR-A}, the convex set $F$ is a joint 
numerical range:
\begin{align*}
F 
&=\{\scp{\rho,A_i}_{i=1}^k\colon \rho\in\cD_{P\M_nP}\}
=\{\scp{\rho,PA_iP}_{i=1}^k\colon \rho\in\cD_{P\M_nP}\}\\
&=W_{P\M_nP}(PA_1P,\ldots,PA_kP)
\end{align*}
\item
The matrix size $n$ can be decreased at the expense of a unitary similarity. Let 
$U$ be a unitary $n\times n$ matrix such that $Q:=UPU^\ast$ has the block diagonal 
form $Q=\id_r\oplus 0$. The matrix $PAP$ is unitarily similar to 
$UPAPU^\ast=QU A U^\ast Q=UAU^\ast[r]\oplus 0$ for every $A\in\M_n$. Moreover, the 
trace-preserving *-algebra isomorphism
\[
P\M_nP\to\M_r,
\quad
A\mapsto (UAU^\ast)[r]
\]
restricts to an isometry $\cD_{P\M_nP}\to\cD_{\M_r}$. Hence, 
\begin{align*}
F 
&=\{\scp{\rho,PA_iP}_{i=1}^k\colon \rho\in\cD_{P\M_nP}\}
=\{\scp{U\rho U^\ast[r],UA_iU^\ast[r]}_{i=1}^k\colon \rho\in\cD_{P\M_nP}\}\\
&=\{\scp{\sigma,UA_iU^\ast[r]}_{i=1}^k\colon \sigma\in\cD_{\M_r}\}
=W\bigl((UA_1U^\ast)[r],\ldots,(UA_kU^\ast)[r]\bigr).
\end{align*}
\item
The number $k$ of matrices can be reduced to the dimension $d=\dim(F)$ of $F$ 
at the cost of an affine transformation with coefficients $t_{ij}\in\R$, $i=1,\dots,l$, 
$j=0,\dots,k$,
\begin{equation}\label{eq:affT}\textstyle
T:\R^k\to\R^l,
\quad
(x_1,\ldots,x_k)\tp
\mapsto
(t_{i0}+\sum_{j=1}^kt_{ij}x_j)_{i=1}^l.
\end{equation}
Let
\begin{equation}\label{eq:affA}\textstyle
\hat{A}_i:=t_{i0}\id_n+\sum_{j=1}^kt_{ij}A_j,
\quad 
i=1,\ldots,l.
\end{equation}
Then for every $\rho\in\cD_n$ we have
\begin{equation}\label{eq:Tw=w}
T\circ w_{A_1,\dots,A_k}(\rho)=w_{\hat{A}_1,\dots,\hat{A}_l}(\rho).
\end{equation}
In particular, $T(W)=W(\hat{A}_1,\ldots,\hat{A}_l)$. Let $T:\R^k\to\R^k$ be a rotation 
such that for every $i=d+1,\dots,k$ the number $\langle\rho,\hat{A}_i\rangle=c_i$ is 
constant for all $\rho\in\cD_{P\M_nP}$. Then
\begin{align}\label{eq:affinityTF-Wc}
T(F)
&\textstyle
=\{\langle \rho,\hat{A}_i\rangle_{i=1}^k\colon \rho\in\cD_{P\M_nP}\}
=\{\langle\rho,\hat{A}_i\rangle_{i=1}^d\colon \rho\in\cD_{P\M_nP}\}
\times\{c\}\\\nonumber
&=W\bigl((U\hat{A}_1U^\ast)[r],\ldots,(U\hat{A}_d U^\ast)[r]\bigr)
\times\{c\},
\end{align}
where $c=(c_{d+1},\dots,c_k)\tp$ and $U$ is the unitary from part~(b) 
above. Hence, $F$ is isometric to the joint numerical range 
$W\bigl((U\hat{A}_1U^\ast)[r],\ldots,(U\hat{A}_d U^\ast)[r]\bigr)$.
\item
It often suffices to study $F$ though an affinity rather than an isometry. Let $\cV$ be 
the linear span of the matrices $\id_r,(UA_1U^\ast)[r],\ldots,(UA_kU^\ast)[r]$ introduced 
in part~(b) above and let $B_1,\dots,B_l\in\cH_r$ satisfy
\[
\spn\{\id_r,B_1,\dots,B_l\}=\cV.
\]
Then the joint numerical ranges $F$ and $W(B_1,\dots,B_l)$ are affinely isomorphic. Indeed,  
by part~(b) above we have
$F=W\bigl((UA_1U^\ast)[r],\ldots,(UA_kU^\ast)[r]\bigr)$ and the maps
\begin{align*}
w_{\id_r,(UA_1U^\ast)[r],\ldots,(UA_kU^\ast)[r]}&:
 \pi_\cV(\cD_r)\to \{1\}\times W\bigl((UA_1U^\ast)[r],\ldots,(UA_kU^\ast)[r]\bigr)\\
w_{\id_r,B_1,\dots,B_l}&:\pi_\cV(\cD_r)\to \{1\}\times W(B_1,\dots,B_l)
\end{align*}
are affinities by equation \eqref{eq:aff-pi-1W}. 
\end{enumerate}
\end{remark}
Let $C$ be a convex subset of a Euclidean space $E$. The \emph{normal cone} 
of $C$ at $x\in C$ is the closed convex cone
\[
N_C(x):=\{u\in E\mid \forall y\in C\colon \scp{y-x,u}\geq 0 \}.
\]
We put $N_C(\emptyset):=E$. The \emph{normal cone} of $C$ at a nonempty convex 
subset $G$ of $C$ is 
\[\textstyle
N_C(G):=\bigcap_{x\in G}N_C(x).
\]
Note that $N_C(G)=N_C(x)$ holds for every point $x$ in the relative interior%
\footnote{%
The \emph{relative interior} \cite{Rockafellar1970} of a convex set $C$ is the 
interior of $C$ in the topology of its affine hull.} 
of $G$, see \cite[Definition~4.3]{Weis2012}. The set of 
\emph{normal cones} of $C$ is defined as 
$\{N_C(G):\text{$G\subset C$ is convex }\}$. 
\par
By 
{\cite[Propositions~4.7 and~4.8]{Weis2012} we have the following.
\par
\begin{proposition}\label{pro:normal-cones}
The set of normal cones of a convex subset $C$ of a Euclidean space $E$ is a 
complete lattice with respect to the partial order by inclusion. If $K\subset L\neq E$ 
are normal cones of $C$, then $K$ is a face of $L$. If $C$ is not a singleton, 
then $F\mapsto N_C(F)$ is an antitone lattice isomorphism from the set of 
exposed faces to the set of normal cones of~$C$.
\end{proposition}
A point $x\in C$ is a \emph{corner point} of $C$ if $N_C(x)$ has full dimension
$\dim(N_C(x))=\dim(E)$. It is well known, see \cite[Proposition~2.5]{BindingLi1991}, 
that every corner point $p=(p_1,\dots,p_k)\tp\in\R^k$ of $W$ is contained in the 
\emph{joint spectrum} 
\[
\{(\lambda_1,\dots,\lambda_k)\tp\in\R^k
\mid \exists \ket{\varphi}\in\C^n\setminus\{0\}, \forall i=1,\dots,k 
\colon A_i\ket{\varphi}=\lambda_i\ket{\varphi}\}
\]
of the hermitian matrices $A_1,\dots,A_k\in\M_n$. This shows that there is 
a nonzero $\ket{\varphi}\in\C^n$ such that 
$A_i\ket{\varphi}=A_i^\ast\ket{\varphi}=p_i\ket{\varphi}$, $i=1,\dots,k$,
which means that $\ket{\varphi}$ is a \emph{normal eigenvector} of $A_i$, 
$i=1,\dots,k$. Hence \cite[p.~123]{HornJohnson2012}, there is a unitary 
$n\times n$ matrix $U$, such that
\[
UA_iU^\ast
=p_i\oplus B_i
=\begin{bmatrix}
p_i & 0\\
0 & B_i
\end{bmatrix},
\quad i=1,\dots,k
\]
for certain hermitian $(n-1)\times(n-1)$ matrices $B_1,\dots,B_k$. We arrive at the same 
conclusion in Proposition~\ref{pro:corner-points}. Let $\C\id_n:=\{z\id_n \mid z\in\C\}$ and
\begin{align*}
\C P\oplus P'\M_nP'
&:=\{z P + P'AP' \mid z\in\C, A\in\M_n\},\\
P\M_nP\oplus P'\M_nP'
& :=\{ PAP + P'AP' \mid A\in\M_n\},
\end{align*}
for every orthogonal projection $P\in\cP_n$, $P\neq0,\id_n$, where 
$P':=\id_n-P$.
\par
\begin{proposition}\label{pro:corner-points}
Let $p$ be a corner point of $W$ and let $P\in\cP_n$ be the projection 
associated to $\{p\}$. If $P=\id_n$, then $A_1,\dots,A_k\in\C\id_n$. 
Otherwise $A_1,\dots,A_k\in\C P\oplus P'\M_nP'$.
\end{proposition}
\begin{proof}
We replace $W$ with the projection $\cC:=\pi_\cV(\cD_n)$ of $\cD_n$ onto the span $\cV$ 
of $\id_n,A_1,\dots,A_k$. The map $w_{\id_n,A_1,\dots,A_k}:\cH_n\to\R^{k+1}$ restricts 
to the affinity $w_{\id_n,A_1,\dots,A_k}:\cC\to \{1\}\times W$ by equation 
\eqref{eq:aff-pi-1W}. The matrix $M:=w|_\cC^{-1}(1,p)$ is therefore a corner point of 
$\cC$. Since $M$ is an exposed point of $\cC$, equations
(3.10) and (4.5) in \cite{Weis2018} show that the normal cone of $\cC$ at $M$ in $\cV$ is
\[
N_\cC(M)=\{A\in\cH_n\mid P\leq P_-(A)\}\cap\cV,
\]
where $P_-(A)$ is the spectral projection of $A\in\cH_n$ corresponding to the 
smallest eigenvalue. As shown in \cite{Weis2018}, formula (5.3),
we have
\[
N_\cC(M) 
=\R\id_n+(P'\M_n^+P'\cap\cV),
\]
where $\M_n^+$ denotes the cone of positive semidefinite $n\times n$ matrices.
Since $M$ is a corner point, $\cV$ is included in the linear span of $N_\cC(M)$. 
This shows that $\cV\subset\R\id_n+(P'\cH_nP')$ and proves the claim.
\end{proof}
The next lemma describes the joint numerical range of a reducible tuple of 
hermitian matrices. We skip its straightforward proof.
\begin{lemma}\label{lem:convWs}
Let $P\in\cP_n$, $P\neq0,\id_n$, and let $A_1,\dots,A_k\in P\M_nP\oplus P'\M_nP'$. 
Then
\[
W=\conv[W_{P\M_nP}(PA_1P,\dots,PA_kP)
\cup
W_{P'\M_nP'}(P'A_1P',\dots,P'A_kP')].
\]
If, in addition, $p:=(p_1,\dots,p_k)\tp\in\R^k$, and $A_iP=p_iP$, $i=1,\dots,k$, then 
\[
W=\conv[\{p\}\cup W_{P'\M_nP'}(P'A_1P',\dots,P'A_kP')].
\]
\end{lemma}

\begin{corollary}\label{cor:convWcorner}
Assume that $W$ is not a singleton and let $p$ be a corner point of $W$. Then $W$ is 
the convex hull of $p$ and the joint numerical range $W(B_1,\dots,B_k)$ of $k$ hermitian 
$s\times s$ matrices $B_1,\dots,B_k$ for some $s<n$. If $\{p\}$ is a rank-$r$ face, then
$s=n-r$ is a valid choice.
\end{corollary}

\begin{proof}
By Proposition~\ref{pro:corner-points}, we have $A_1,\dots,A_k\in\C P\oplus P'\M_nP'$,
where $P\neq 0$ is the projection associated with $\{p\}$. This implies 
$A_iP=p_iP$, $i=1,\dots,k$, where $p=(p_1,\dots,p_k)\tp\in\R^k$. Lemma~\ref{lem:convWs}
then shows that $W$ is the convex hull of $p$ and  $W_{P'\M_nP'}(P'A_1P',\dots,P'A_kP')$.
Remark~\ref{rem:faces-JNR}~(a) and~(b) complete the proof.
\end{proof}
An application of Lemma~\ref{lem:convWs} or Corollary~\ref{cor:convWcorner} 
requires describing the faces of the convex hull of two convex sets.
\par
\begin{lemma}\label{lem:faces-convex-hull}
Let $F$ be a face of the convex hull of two convex subsets $Y_1,Y_2$ 
of a real vector space. Then $F\cap Y_i$ is a face of $Y_i$, $i=1,2$, and 
$F=\conv((F\cap Y_1)\cup(F\cap Y_2))$.
\end{lemma}
\begin{proof}
Let $X:=\conv(Y_1\cup Y_2)$. It is clear that $F\cap Y_i$ is a face of $Y_i$ because 
$F$ is a face of $X$ and $Y_i$ is a convex subset of $X$, $i=1,2$. The inclusion 
$F\supset\conv((F\cap Y_1)\cup(F\cap Y_2))$ is clear since $F$ is convex. Conversely, 
let $x\in F$. Then there exist $y_1\in Y_1$, $y_2\in Y_2$, and $s\in[0,1]$ such 
that $x=(1-s)y_1 + sy_2$. If $x=y_i$ for an $i=1,2$, then $x\in F\cap Y_i$ holds.
Otherwise $s\in(0,1)$ and $y_1,y_2\in F$ follows because $F$ is a face of $X$.
\end{proof}
We recall the well-known relationship between Bonsall and Duncan's and the standard
definition of the joint numerical range \eqref{eq:JNR-BonsallDuncan} and 
\eqref{eq:JNR-standard}, respectively, in the setup of this section. Note 
that $\braket{\varphi}{A\varphi}=\tr(\ketbra{\varphi}A)=\scp{\ketbra{\varphi},A}$
holds for every $A\in\M_n$ and $\ket{\varphi}\in\C^n$.
We denote the set of extreme points of a convex set $C$ by $\ext C$. 
Let $\range(A):=\{A\ket{\varphi}\colon \ket{\varphi}\in\C^n\}$ denote the \emph{range} 
of $A\in\M_n$.
\par
\begin{lemma}\label{lem:pureJNR}
Let $P\in\cP_n$ be a nonzero projection and let $F:=w(\cD_{P\M_nP})$. 
Then 
\[
F=\conv[w(\ext\cD_{P\M_nP})].
\]
If $\dim(F)\leq 2$ or if $\dim(F)=3$ and $\rk(P)\geq 3$, then 
$F=w(\ext\cD_{P\M_nP})$.
\end{lemma}
\begin{proof}
As a preliminary step, we recall that
\[
\ext\cD_{P\M_nP}=\{\ketbra{\varphi}\colon\ket{\varphi}\in\range(P), \braket{\varphi}=1\}.
\]
The extreme points of $\cD_n$ are parametrized by the map
\[
\{\ket{\varphi}\in\C^n\mid\braket{\varphi}=1\}
\to\ext\cD_n,
\quad
\ket{\varphi}\mapsto\ketbra{\varphi},
\]
see, e.g., \cite[Section~5.1]{BengtssonZyczkowski2017}. Clearly, a hermitian matrix 
$A\in\cH_n$ is contained in $P\M_nP$ if and only if $\range(A)\subset\range(P)$. Since 
$\cD_{P\M_nP}$ is a face of $\cD_n$, a point in $\cD_{P\M_nP}$ is an extreme point of 
$\cD_{P\M_nP}$ if and only if it is an extreme point of $\cD_n$.
\par
Regarding the first assertion, observe that $w(\cD_{P\M_nP})$ is a linear image of the 
compact convex set $\cD_{P\M_nP}$. The claim is true because preimages of faces of $F$ 
under $w|_{\cD_{P\M_nP}}$ are faces of $\cD_{P\M_nP}$ and because every compact convex 
set is the convex hull of its extreme points (a statement known as the Minkowski-Steinitz 
theorem \cite{Schneider2014}). 
\par
Let $d=\dim(F)$ and $r=\rk(P)$. There are a rotation $T:\R^k\to\R^k$, hermitian 
$B_1,\dots,B_d\in\cH_r$, and $c\in\R^{k-d}$ such that 
$F=T^{-1}[W(B_1,\ldots,B_d)\times\{c\}]$, as described in equation 
\eqref{eq:affinityTF-Wc}. Similarly, 
$w(\ext\cD_{P\M_nP})=T^{-1}[\JNR(B_1,\ldots,B_d)\times\{c\}]$.
It is well known \cite{Au-YeungPoon1979,MaierNetzer2024} that $\JNR (B_1,\ldots,B_d)$ 
is convex in case $d\leq 2$ and in case $d=3$ and $r\geq 3$. 
This proves the second assertion.
\end{proof}
The assertion $W=\conv\JNR$ connecting the two definitions 
\eqref{eq:JNR-BonsallDuncan} and \eqref{eq:JNR-standard} of the joint numerical 
range is included in Lemma~\ref{lem:pureJNR} for $P=\id_n$.
\par
%
%
\section{Numbers of non-elliptic faces}
\label{sec:non-elliptic}
We describe all possible numbers that occur for the different types 
of non-elliptic faces of the joint numerical range $W:=W(A_1,A_2,A_3)$ of 
three hermitian $4\times 4$ matrices $A_1,A_2,A_3$. We assume $\dim(W)=3$.
\par
Observe that every non-elliptic face of $W$ is a rank-3 face
because rank-2 faces are linear images of a three-dimensional Euclidean ball.
\par
\begin{lemma}\label{lem:int2non-elliptic}
If a rank-3 face of $W$ differs from a non-elliptic face of $W$, then 
both faces have dimension two. Their intersection is a one-dimensional face 
of each of them and an exposed face of $W$.
\end{lemma}
\begin{proof}
Let $F$ denote a non-elliptic face and $G$ a rank-3 face. There exist orthogonal projections 
$P,Q\in\cP_4$ such that $\cD_{P\M_4P}=w|_{\cD_4}^{-1}(F)$ and $\cD_{Q\M_4Q}=w|_{\cD_4}^{-1}(G)$
by Theorem~\ref{thm:PF-iso}. Since $P$ and $Q$ have rank three, the intersection of their 
ranges is a two-dimensional subspace of $\C^4$. Let $R\in\cP_4$ denote the orthogonal 
projection onto this subspace. Then the preimage of $S:=F\cap G$ is the three-dimensional ball
$\cD_{R\M_4R}=w|_{\cD_4}^{-1}(S)$. Therefore, $S$ cannot be a singleton, as this would lead to 
the contradiction that $F$ is an ellipse by \cite[Lemma 4.5]{Szymanski-etal2018}. This shows 
that $S$ is a segment. As $F$ and $G$ are exposed faces of $W$, so is $S=F\cap G$.
\end{proof}

More details on the use of \cite[Lemma 4.5]{Szymanski-etal2018} are explained 
in the following remark.
\par
\begin{remark}[The ellipticity criterion for bordered matrices]\label{rem:bordered}
Let $R\leq P$ be the projections from the proof of Lemma~\ref{lem:int2non-elliptic}, where 
$\rk(P)=3$ and $\rk(R)=2$. Then the joint numerical range 
$F=w_{A_1,A_2,A_3}(\cD_{P\M_4P})$ is affinely isomorphic to $W(B_1,B_2)$ for some hermitian 
$3\times 3$ matrices $B_1,B_2$ by Remark~\ref{rem:faces-JNR}~(d).
We can choose $B_1,B_2$ such that the linear span of $\id_3,B_1,B_2$ is that of 
$\id_3,(UA_1U^\ast)[3],(UA_2U^\ast)[3],(UA_3U^\ast)[3]$, where $U\in\M_4$ is a unitary 
that diagonalizes $P$ to $UPU^\ast=\id_3\oplus 0$.
Let $R'\in\cP_3$ be such that $URU^\ast=R'\oplus 0$.
If the image of the three-dimensional ball $\cD_{R\M_4R}$ under 
$w_{A_1,A_2,A_3}:\cH_4\to\R^3$ is an exposed point of $F$, then the image of 
$\cD_{R'\M_3R'}$ under $w_{B_1,B_2}:\cH_3\to\R^2$ is an exposed point of $W(B_1,B_2)$.
Hence, \cite[Lemma 4.5]{Szymanski-etal2018} proves that up to a unitary similarity of 
$B_1$ and $B_2$ and an affinity acting on $W(B_1,B_2)$ as in Remark~\ref{rem:faces-JNR}~(c), 
one has
\[
B_1=\diag(1,0,0)
\quad\text{and}\quad
B_2=\begin{bmatrix} 0 & 0 & 1\\0 & 0 & 0\\1 & 0 & 0\end{bmatrix}.
\]
These are examples of \emph{bordered matrices}: Every matrix entry
that lies outside the first column and row vanishes.
By \cite[Proposition~3.2]{AdamMaroulas2002} the joint numerical range of any $k$ linearly 
independent bordered hermitian $n\times n$ matrices is a $k$-dimensional 
ellipsoid in $\R^k$.
Another test that corroborates that $W(B_1,B_2)$ is an ellipse is described 
in \cite[Section~4]{BrownSpitkovsky2004}.
\end{remark}
Lemma~\ref{lem:int2non-elliptic} restricts the possible shapes of rank-3 faces of $W$
significantly. Rank-3 faces that are ellipses, segments, or singletons can only exist 
if $W$ has no non-elliptic faces. Otherwise, if $W$ has a non-elliptic face, then any 
other rank-3 face is also a non-elliptic face. 
\par
Lemma~\ref{lem:int2non-elliptic} is no longer true if we modify its statement by 
replacing `non-elliptic' with `rank-3', as it is witnessed by the following example.
\par
\begin{example}\label{ex:2rank3ellipses}
The joint numerical range of the hermitian $3\times 3$ matrices 
$\widehat{A}_1,\widehat{A}_2,\widehat{A}_3$ from \cite[Example~6.3]{Szymanski-etal2018} has 
two ellipses on its boundary that intersect at the point $(-1,0,0)\tp$.
We have $W(A_1,A_2,A_3)=W(\widehat{A}_1,\widehat{A}_2,\widehat{A}_3)$ 
if we extend $\widehat{A}_1,\widehat{A}_2,\widehat{A}_3$ to the hermitian $4\times 4$ block diagonal 
matrices $A_1:=\widehat{A}_1\oplus(-1)$, $A_2:=\widehat{A}_2\oplus0$, $A_3:=\widehat{A}_3\oplus0$. 
Then the two ellipses are rank-3 faces of $W(A_1,A_2,A_3)$ and they intersect at a single point.
\end{example}
Theorem~\ref{thm:non-elliptic} follows essentially from Lemma~\ref{lem:int2non-elliptic}. 
Just one case needs to be excluded in Lemma~\ref{lem:not0003}. This is prepared by the 
following lemma, which is also used in the proof of Lemma~\ref{lem:3nonexp-segments}.
\par
\begin{lemma}\label{lem:int3non-elliptic}
The intersection of any three mutually distinct non-elliptic faces of 
$W$ is a singleton which is a corner point of $W$. In particular, the three faces have 
a point in common.
\end{lemma}
\begin{proof}
Let $F_1,F_2,F_3$ be three mutually distinct non-elliptic faces of $W$.
Lemma~\ref{lem:int2non-elliptic} shows that $S_2:=F_1\cap F_2$ and $S_3:=F_1\cap F_3$ are 
one-dimensional exposed faces of $W$.
\par
The intersection $F_1\cap F_2\cap F_3$ is empty or a singleton. 
Indeed, if $S_2$ were a face of $F_3$, then by Proposition~\ref{pro:normal-cones}
the normal cone of $S_2$ would have three mutually distinct faces $N_W(F_i)$, $i=1,2,3$ 
that are strictly included in $N_W(S_2)$ and differ from $\{0\}$. This would imply 
$\dim N_W(S_2)=3$, contradicting the fact that $S_2$ is a segment. 
Therefore $F_1\cap F_2\cap F_3=S_2\cap F_3$ is a face of $S_2$ 
that is strictly included in $S_2$:
a singleton or the empty set.
\par
Now we show that $F_1\cap F_2\cap F_3\neq\emptyset$.
By what is shown in the preceding paragraph, the segments $S_2$ and $S_3$
are different. Hence, the non-elliptic face $F_1$, of which $S_2$ and
$S_3$ are faces, is either a droplet or a triangle.
In either case, $S_2$ and $S_3$ have a point in common.
\par
Let $p$ be the unique point in $F_1\cap F_2\cap F_3$. The chain of exposed 
faces $F_1\supset S_2\supset\{p\}$ and Proposition~\ref{pro:normal-cones} provide
the strict inclusions of normal cones
\[
N_W(F_1)\subset N_W(S_2)\subset N_W(p).
\]
Since the ray $N_W(F_1)$ is a face of $N_W(S_2)$, which is a face of 
$N_W(p)$, it follows that the dimension of $N_W(p)$ is three.
\end{proof}
\begin{lemma}\label{lem:not0003}
If $W$ has three mutually distinct triangular faces, 
then $W$ is a tetrahedron.
\end{lemma}
\begin{proof}
Lemma~\ref{lem:int2non-elliptic} and Lemma~\ref{lem:int3non-elliptic} show that 
the pairwise intersections of the triangles are three mutually distinct 
segments. The following figure depicts the triangles when one of them, the 
triangle $bde$, is glued to the other ones along shared sides.
\vskip.5\baselineskip
\hspace*{\fill}
\begin{tikzpicture}
 \coordinate (1) at (-2,0); 
 \coordinate (2) at (0,0); 
 \coordinate (3) at (2,0); 
 \coordinate (4) at (1,1.732); 
 \coordinate (5) at (-1,1.732);     
 \path[draw,thick] 
   (1) node[left] {$a$}
   -- (2) node[below] {$b$}
   -- (3) node[right] {$c$};
 \path[draw,thick] 
   (2)
   -- (4) node[right] {$d$}
   -- (5) node[left] {$e$} 
   -- cycle;
 \draw[thick] (1) -- (5);
 \draw[thick] (3) -- (4);
\end{tikzpicture}
\hspace*{\fill}

\noindent For the three common edges of the triangles intersect at
a point, it is evident that the faces $abe$ and $bcd$ intersect along $[b,a]=[b,c]$ and $a=c$.
\par
The point $b$ is a corner point of $W$ by Lemma~\ref{lem:int3non-elliptic}.
Hence, Corollary~\ref{cor:convWcorner} shows that $W$ is the convex hull 
$W=\conv(\{b\}\cup W')$ of $b$ and the joint numerical range $W'$ of three 
hermitian $3\times 3$ matrices. The triangle $abe$ is a face of $W$ and 
equals the convex hull 
\[
abe=\conv(\{b\}\cup F)
\]
of the corner point $b$ and a face $F$ of $W'$ by 
Lemma~\ref{lem:faces-convex-hull}. The inclusion $W'\subset abe$ is wrong 
because it implies $W=abe$ is a triangle. Hence, $F$ is strictly included 
in $W'$, which means that $F$ is a rank-1 or rank-2 face of $W'$, that 
is to say, a singleton, a segment, or a filled ellipse. Of these three 
choices, only the segment $F=[a,e]$ can fulfill $abe=\conv(\{b\}\cup F)$. 
\par
Similarly, all three segments $[a,e]$, $[e,d]$, and $[d,c]=[d,a]$ are faces 
of $W'$. There is only one joint numerical range $W'$ of three hermitian 
$3\times 3$ matrices with three or more one-dimensional faces,
see 
the introduction of \cite{Szymanski-etal2018}: the triangle.
Hence, $W$ is a tetrahedron.
\end{proof}

Having the above lemmas, we are ready to prove the main classification result
of the paper.
%
\begin{theorem}\label{thm:non-elliptic}
The numbers of non-elliptic faces of the joint numerical range
$W$ belong to one of the following 15 columns.
\[\begin{array}{c|c||c|c|c|c|c|c|c|c|c|c|c|c|c|c|c}
\text{type} & \text{shape} & \multicolumn{15}{c}{\text{count}} \\
\hline\hline 0 & \text{oval}
 & 0 & 1 & 0 & 0 & 0 & 0 & 0 & 0 & 0 & 0 & 0 & 0 & 0 & 0 & 0 \\
\hline 1 & \text{loaf}
 & 0 & 0 & 1 & 2 & 1 & 1 & 0 & 0 & 0 & 0 & 0 & 0 & 0 & 0 & 0 \\
\hline 2 & \text{droplet}
 & 0 & 0 & 0 & 0 & 1 & 0 & 1 & 2 & 3 & 1 & 1 & 2 & 0 & 0 & 0 \\
\hline 3 & \text{triangle}
 & 0 & 0 & 0 & 0 & 0 & 1 & 0 & 0 & 0 & 1 & 2 & 1 & 1 & 2 & 4 \\
\hline\hline\multicolumn{2}{r||}{\text{example no.}}&
0 & 1 & 2 & 3 & 4 & 5 & 6 & 7 & 8 & 9 & 10 & 11 & 12 & 13 & 14
\\\hline\end{array}\]
All 15 tuples are attained by respective examples presented in Section~\ref{sec:examples} 
below. Every tuple described by the columns 8--11 or 13--14 implies the existence 
of a corner point of $W$. 
\end{theorem}
\begin{proof}
The table can be derived from the fact that any two distinct non-elliptic faces
intersect along a segment (Lemma~\ref{lem:int2non-elliptic}).
\par
Let $a_i\in\mathbb Z_{\geq 0}$ denote the number of non-elliptic faces of type $i$
of $W$, $i=0,1,2,3$. A face of type $0$ cannot intersect any other non-elliptic face,
hence $a_0>0$ implies $a_1=a_2=a_3=0$ and $a_0=1$. Similarly, if $a_0=0$ and $a_1\geq 1$, 
then $a_1+a_2+a_3\leq 2$. Also, if $a_0=a_1=0$ and $a_2\geq 1$, then $a_2+a_3\leq 3$.
Finally, $a_3\leq 4$. The numbers $a_0=a_1=a_2=0$ and $a_3=3$ are excluded by 
Lemma~\ref{lem:not0003}.
\par
The existence of a corner point follows from Lemma~\ref{lem:int3non-elliptic} for the 
columns 8, 10, 11, and~14, and from Theorem~\ref{thm:3segments} for the columns 9 and~13.
\end{proof}
%
%
%
\section{Examples} \label{sec:examples}

Here we present examples of hermitian $4\times 4$ matrices $A_1,A_2,A_3$ 
with joint numerical ranges $W=W(A_1,A_2,A_3)$ having all possible numbers of 
non-elliptic faces listed in Theorem~\ref{thm:non-elliptic}. Figure \ref{fig:all-examples} shows $W$ of all the examples with non-elliptic faces highlighted. Proving that each rank-3 
face has the specified non-elliptic shape and that $W$ has no further rank-3 faces
is a straightforward application of the methods of
Remark~\ref{rem:detect-faces} and Remark~\ref{rem:deciding-shape}, 
and is therefore omitted.
\par
\begin{definition}
Let $u_0,u_1,u_2,u_3\in\R$, $A:=u_0\id_4+u_1A_1+u_2A_2+u_3A_3$, and let $A\{i,j\}$ ($i\neq j$)
be the principal minor of $A$ that is the determinant of the submatrix obtained by deleting 
all the rows and columns except for those with the indices $i,j\in\{1,2,3,4\}$. 
Similarly, $A\{i,j,k\}$ is a principal minor of order three. If $A$ has rank one, then
we call $(u_0,u_1,u_2,u_4)$ a \emph{rank-1 tuple} and we refer to the kernel projection 
of $A$ as the \emph{kernel projection} of $(u_0,u_1,u_2,u_4)$. 
\end{definition}
\begin{remark}[Controlling rank-3 and non-elliptic faces]\label{rem:detect-faces} 
A necessary condition for the existence of a non-elliptic face is that a rank-1 tuple 
$(u_0,u_1,u_2,u_3)$ exists. More specifically, every non-elliptic face $F$ of $W$ is a 
rank-3 face of $W$. It has the form $F=w(\cD_{P\M_dP})$, where $P\in\cP_n$ is the 
rank-3 projection associated to $F$. Since $F$ is an exposed face of $W$, there are 
$u_1,u_2,u_3\in\R$ such that $P$ is the spectral projection of $u_1A_1+u_2A_2+u_3A_3$ 
with respect to the smallest eigenvalue by Theorem~\ref{thm:PF-iso}. Equivalently, 
there is a rank-1 tuple $(u_0,u_1,u_2,u_4)$ whose kernel projection is $P$. The 
rank-one condition is rather easy to verify in terms of 
vanishing of the six order two principal minors 
\[
A\{1,2\},\quad
A\{1,3\},\quad
A\{1,4\},\quad
A\{2,3\},\quad
A\{2,4\},\quad
A\{3,4\}
\]
and the four order three principal minors
\[
A\{1,2,3\},\quad
A\{1,2,4\},\quad 
A\{1,3,4\},\quad
A\{2,3,4\}
\]
of $A=u_0\id_4+u_1A_1+u_2A_2+u_3A_3$, see \cite{Dickson1959},
Theorem~16 of Chapter~IV.
\end{remark}
\begin{remark}[Deciding on the shape of a rank-3 face]\label{rem:deciding-shape}
Let $P\in\cP_n$ be a rank-3 projection and let $F:=w(\cD_{P\M_4P})$. For example, 
if $(u_0,u_1,u_2,u_3)$ is a rank-1 tuple, whose kernel projection is $P$, then $F$ 
is a rank-3 face of $W$.
\begin{enumerate}[(a)]\itemsep=4pt
\item
\textbf{Employing $3\times 3$ matrices.} Let $B_1,B_2$ be hermitian $3\times 3$ 
matrices such that the linear span of the matrices $\id_3,B_1,B_2$ is the same 
as that of 
\[
\id_3,\quad
(UA_1U^\ast)[3],\quad
(UA_2U^\ast)[3],\quad
(UA_3U^\ast)[3],
\]
where $U\in\M_4$ is a unitary such that $UPU^\ast=\id_3\oplus 0$. Then 
Remark~\ref{rem:faces-JNR}~(d) guarantees that $F$ is affinely isomorphic to the joint 
numerical range $W(B_1,B_2)$.
\item
\textbf{Invoking the numerical range.}
One can also use the numerical range $\JNR (B)$ of 
\[
B:=B_1+\ii B_2
\] 
to decide on the shape of $F\cong W(B_1,B_1)$ because $W(B_1,B_1)=\JNR (B)$ holds
as recalled in \eqref{eq:NR}. We employ Kippenhahn's classification to decide on 
the shape of $\JNR (B)$.
\item
\textbf{Deciding unitary reducibility.}
A square matrix is \emph{unitarily reducible} if it is unitarily similar to a block 
diagonal matrix with at least two proper blocks. Otherwise the matrix is 
\emph{unitarily irreducible}.
We claim that the $3\times 3$ matrix $B$ is unitarily irreducible if and only if the 
positive semidefinite matrix
\[
S:=[B_1,B_2]^\ast[B_1,B_2] 
+ [B_1^2,B_2]^\ast[B_1^2,B_2]
+ [B_1,B_2^2]^\ast[B_1,B_2^2]
+ [B_1^2,B_2^2]^\ast[B_1^2,B_2^2]
\]
is positive definite, where $[X,Y]:=XY-YX$ is the commutator of $X,Y\in\cH_3$. The 
Sylvester criterion allows us to decide the positive definiteness of $S$ easily.
\par
We prove our claim. If $B$ is unitarily reducible then it generates a unital *-algebra 
(the smallest *-algebra on $\C^3$ containing $B$ and $\id_3$) that is strictly smaller 
than $\M_3$. The converse is also true since any unital *-algebra on $\C^3$ has a block
diagonal form up to unitary similarity (see, e.g., \cite[Theorem~5.6]{Farenick2001}). 
Clearly, the unital *-algebra generated by $B$ is the set of the values of all 
noncommutative polynomials in the variables $B_1$ and $B_2$ (we identify scalars with 
scalar matrices). According to \cite{AslaksenSletsjoe2009}, Theorem~8, $B$ generates 
the unital *-algebra $\M_3$ if and only if $S$ is positive definite. 
\item 
\textbf{The unitarily reducible case.} 
Let $B$ be unitarily reducible. 
\begin{itemize}
\item
If $B_1$ and $B_2$ commute, then $B$ is normal and $\JNR (B)$ is the convex hull of the 
three eigenvalues of $B$, which is a triangle if and only if the eigenvalues are real 
affinely independent.
\item 
If $B_1$ and $B_2$ do not commute then the curve $p_B=0$ in $\C\P^2$ defined by the polynomial 
\eqref{eq:hyperbolic} is the union of a conic $e$ and a line $\ell$. The Kippenhahn curve 
$C_\R(B)$ is the union of an ellipse $e^\ast$ and a point $p$ in $\R^2$. By substituting the 
equation of $\ell$ into that of $e$, it is easy to compute the two points in $e\cap\ell$. 
They are distinct and real if and only if $p$ lies outside $e^\ast$, if and only if 
$\JNR (B)$ is a droplet.
\end{itemize}
\item 
\textbf{The unitarily irreducible case.}
Let $B$ be unitarily irreducible. Then $\JNR (B)$ is an ellipse, a loaf, or an oval. 
We use conditions from \cite{Keeler-etal1997} to distinguish these shapes.
\begin{itemize}
\item
\textbf{Ellipses.} Let $\lambda_1,\lambda_2,\lambda_3$ be the eigenvalues of $B$. By 
\cite[Theorem~2.3]{Keeler-etal1997}, the Kippenhahn curve $C_\R(B)$ consists of an 
ellipse and a point if and only if 
\begin{itemize}
\item 
the number $\delta:=\tr(B^\ast B)-\sum_{i=1}^3|\lambda_i|^2$ is strictly positive and
\item 
the number 
$\lambda:=\tr(B)+\frac{1}{\delta}(\sum_{i=1}^3|\lambda_i|^2\lambda_i-\tr(B^\ast B^2))$
coincides with at least one of the eigenvalues of $B$.
\end{itemize}
If these conditions are satisfied\footnote{%
We remark that $\lambda$ lies strictly inside the ellipse (not on the boundary of 
$\JNR (B)$) by \cite[Theorem~1.6.6]{HornJohnson1991} because $B$ is 
unitarily irreducible.}, 
then $C_\R(B)$ is the union of $\lambda$ and the ellipse having its foci at two other 
eigenvalues of $B$ and minor axis of length $\sqrt{\delta}$. 
Ellipticity criteria in some other settings can be found in
\cite[Proposition~3.2]{AdamMaroulas2002} and in \cite[Section~4]{BrownSpitkovsky2004},
see Remark~\ref{rem:bordered} above.
\item 
\textbf{Loafs.}
By \cite[Proposition~3.2]{Keeler-etal1997}, the numerical range $\JNR (B)$ is a loaf 
if and only if there are real $u_0,u_1,u_2\in\R$ such that $u_0\id_3+u_1B_1+u_2B_2$ 
has rank one. This matrix has rank one if and only if its three principal minors of 
order two and its determinant are zero (see \cite[p.~79]{Dickson1959}).
\item 
\textbf{Ovals.}
The preceding conditions for $\JNR (B)$ to be an ellipse or a loaf are not just 
sufficient but also necessary. Hence, their failure proves that $\JNR (B)$ is 
an oval.
\end{itemize}
\end{enumerate}
\end{remark}
%
The remainder of this section presents examples.
They are all depicted in Figure~\ref{fig:all-examples}.
\begin{figure}[p]
\foreach \img/\exlbl/\extxt in {
  tE0/ex-0/E0,    tE1/ex-1/E1,    tE2/ex-2/E2,    tE3/ex-3/E3,
  tE4/ex-4/E4,    tE5/ex-5/E5,    tE6/ex-6/E6,    tE7a/ex-7a/E7a,
  tE7b/ex-7b/E7b, tE8/ex-8/E8,    tE9/ex-9/E9,    tE10/ex-10/E10,
  tE11b/ex-11/E11,tE12/ex-12/E12,  tE13/ex-13/E13, tE14/ex-14/E14}{%
\begin{minipage}[b]{0.24\textwidth}\centering
    \includegraphics[width=\textwidth]{pictures/\img}\\
    \hyperref[\exlbl]{\extxt}\rule[-1ex]{0pt}{3ex}
\end{minipage} \hfill}
\caption{Exemples belonging to 15 classes of joint numerical ranges of 3
hermitian matrices of order 4 concerning the non-elliptic faces
listed in Theorem~\ref{thm:non-elliptic} and analyzed in Section~\ref{sec:examples}.}
\label{fig:all-examples}
\end{figure}
\newcommand{\eblock}[2][4cm]{\begin{minipage}[t]{\textwidth-#1}
#2\rule[-1ex]{0pt}{1.9ex}\end{minipage}\ignorespaces}
\begin{ourexample}[0]\label{ex-0}
No faces: $a_0=a_1=a_2=a_3=0$\\
This is the generic case (see \cite[Proposition~4.9]{Gutkin-etal2004}
and \cite[Theorem~4.2]{Szymanski-etal2018}). Exemplary matrices:
\[ 
\begin{bmatrix}
 0 & 1 & 0 & 0 \\
 1 & 0 & \ii & 0 \\
 0 & -\ii & 0 & 0 \\
 0 &  0 & 0 & 0 
\end{bmatrix}, 
\begin{bmatrix}
 0 & 0 & 0 & 0 \\
 0 & 0 & 1 & 0 \\
 0 & 1 & 0 & \ii \\
 0 & 0 & -\ii & 0 
\end{bmatrix},
\begin{bmatrix}
 0 & 0 & 0 & \ii \\
 0 & 0 & 0 & 0 \\
 0 & 0 & 0 & 1 \\
 -\ii & 0 & 1 & 0
\end{bmatrix}
\]
There are no rank-1 tuples.
\end{ourexample}
\begin{ourexample}[1]\label{ex-1}
Single oval face: $a_0=1$ and $a_1=a_2=a_3=0$\\
Exemplary matrices:
\[ 
\begin{bmatrix}
 0 & \ii & \ii & -2 \\
 -\ii & 0 & \ii & 2 \\
 -\ii & -\ii & 0 & 2 \ii \\
 -2 & 2 & -2 \ii & 0 
\end{bmatrix},
\begin{bmatrix}
 -2 & 1 & 1 & -2 \ii \\
 1 & 1 & 1 & 2 \ii \\
 1 & 1 & 1 & 2 \\
 2 \ii & -2 \ii & 2 & 0
\end{bmatrix},
\,\diag(0,0,0,-4)
\]
Rank-1 tuple: $(0, 0, 0, 1)$
\par
Note that examples with one particular face can be constructed by extending
a $3\times 3$ matrix (pair of hermitian matrices) with appropriate numerical range. 
See also Examples~\hyperref[ex-2]{E2}, \hyperref[ex-6]{E6} and \hyperref[ex-12]{E12}.
\end{ourexample}
\begin{ourexample}[2]\label{ex-2}
Single loaf face: $a_1=1$, $a_0=a_2=a_3=0$\\
Exemplary matrices:
\[ 
\begin{bmatrix}
 0 & 0 & 0 & 1 \\
 0 & 0 & 0 & -1 \\
 0 & 0 & -\sqrt{2} & -\ii \\
 1 & -1 & \ii & 0 
\end{bmatrix},
\begin{bmatrix}
 -\frac{1}{\sqrt{2}} & 0 & \frac{\ii}{\sqrt{2}} & -\ii \\
 0 & \frac{1}{\sqrt{2}} & \frac{\ii}{\sqrt{2}} & \ii \\
 -\frac{\ii}{\sqrt{2}} & -\frac{\ii}{\sqrt{2}} & 0 & 1 \\
 \ii & -\ii & 1 & 0 \end{bmatrix},
\,\diag(0,0,0,2)
\]
Rank-1 tuple: $(0, 0, 0, 1)$
\par
This example is constructed in an analogous way as Example~\hyperref[ex-1]{E1}.
\end{ourexample}
\begin{ourexample}[3]\label{ex-3}
Two loaf faces: $a_1=2$ and $a_0=a_2=a_3=0$\\
Exemplary matrices are
\[
\diag(1,1,1,0), 
F\diag(1,1,1,0)F^*,
\frac{1}{2}
\begin{bmatrix}
 0 & 1 & 0 & -1\\
 1 & 0 & 0 & 0\\ 
 0 & 0 & 0 & 0\\
 -1 & 0 & 0 & 0
\end{bmatrix},
\]
where $F=(F_{i,j})_{i,j=1}^4=\frac{1}{2}(\ii^{(i-1)(j-1)})_{i,j=1}^4$ 
is the unitary Fourier matrix of order $4$.\\
Rank-1 tuples: $(-1, 1, 0, 0)$ and $(-1, 0, 1, 0)$
\end{ourexample}
\begin{ourexample}[4]\label{ex-4}
A loaf and a droplet face: $a_1=a_2=1$ and $a_0=a_3=0$\\
Exemplary matrices:
\[ 
\diag(2,0,0,0),
\begin{bmatrix}
 0 & 0 & 0 & 0\\
 0 & 0 & 0 & 0\\
 0 & 0 & 1 & \ii\\ 
 0 & 0 & -\ii & 1
\end{bmatrix},
\begin{bmatrix}
	\frac12 & -\frac12 & \frac 12 &0\\
	-\frac12 & 1 &0 & 0\\
	\frac12& 0& 0& 1\\
	0&0&1&0\end{bmatrix}
\]
Rank-1 tuples: $(0,1,0,0)$ and $(0,0,1,0)$
\end{ourexample}
\begin{ourexample}[5]\label{ex-5}
A loaf and a triangular face: $a_1=a_3=1$ and $a_0=a_1=0$\\
Exemplary matrices:
\[ 
\diag(0,0,0,1),
\,\diag(0,0,1,0),
\frac12\begin{bmatrix}
 -1 & 0 & -\ii &0\\
 0 & 1 & 1 & 0\\
 \ii & 1 & 0 & 1\\
 0 & 0 & 1 & 0
\end{bmatrix}
\]
Rank-1 tuples: $(0,1,0,0)$ and $(0,0,1,0)$
\end{ourexample}
\begin{ourexample}[6]\label{ex-6}
Single droplet face: $a_2=1$ and $a_0=a_1=a_3=0$\\
Exemplary matrices:
\[ 
\begin{bmatrix}
 -1 & 1 & 0 & -1 \\
 1 & -1 & 0 & 1 \\
 0 &0  & 1 & \ii \\
 -1 & 1 & -\ii & 0 
\end{bmatrix}, 
\begin{bmatrix}
 0 & -\ii & 0 & -\ii \\
 \ii & 0 & 0 & \ii \\
 0 & 0 & 0 & 1 \\
 \ii & -\ii & 1 & 0 
\end{bmatrix},
\,\diag(0,0,0,-2)
\]
Rank-1 tuple: $(0,0,0,1)$
\par
The appropriate triple of hermitian matrices can be constructed in the same way 
as in the case of other unique rank-3 faces 
(cf.\ Examples~\hyperref[ex-1]{E1} and~\hyperref[ex-2]{E2}).
A different construction is described in Remark~\ref{rem:two1D}~(b).
\end{ourexample}
\begin{ourexample}[7a]\label{ex-7a}
\emph{Yin-Yang configuration of droplet faces}: $a_0=a_1=a_3=0$ and $a_2=2$\\
Exemplary matrices:
\[ 
\begin{bmatrix}
 0 & 0 & 0 & 0\\
 0 & 0 & 0 & 0\\
 0 & 0 & 1 & 1\\
 0 & 0 & 1 & 1
\end{bmatrix},
\begin{bmatrix}
 0 & 1 & 0 & 0\\
 1 & 0 & 0 & 0\\
 0 & 0 & 1 & \ii\\
 0 & 0 & -\ii & 1
\end{bmatrix},
\begin{bmatrix}
 1 & \ii & 0 & 0\\
-\ii & 1 & 0 & 0\\
 0 & 0 & 0 & 0 \\ 
 0 & 0 & 0 & 0
\end{bmatrix}
\]
Rank-1 tuples: $(0,1,0,0)$ and $(0,0,0,1)$
\par
The joint numerical range is the convex hull of two orthogonal circles
with centers $(0,0,1)$ and $(1,1,0)$. The intersection of the two droplet faces 
is the segment with endpoints $(0,0,0)$ and $(0,1,0)$.
\end{ourexample}
\begin{ourexample}[7b]\label{ex-7b}
Another example of two droplets: $a_2=2$ and $a_0=a_1=a_3=0$\\
Exemplary matrices:
\[
\begin{bmatrix}
 -1 & 0 & \sqrt 2 & 0\\
 0 & 1 & 0 & 0\\
 \sqrt 2 & 0 & 0 & 0\\
 0 & 0 & 0 & 1
\end{bmatrix},
\begin{bmatrix}
 -1 & 0 & -\sqrt 2 & 0\\
 0 & 1 & 0 & 0\\
 -\sqrt 2 & 0 & 0 & 0\\
 0 & 0 & 0 & 1
\end{bmatrix},
\begin{bmatrix}
 0 & 2 & 0 & 0\\
 2 & 0 & 0 & 0\\
 0 & 0 & 0 & 0\\
 0 & 0 & 0 & 1
\end{bmatrix}
\]
Rank-1 tuples: $(-1,1,0,0)$ and $(-1,0,1,0)$
\par
This example was constructed by adding a corner point to a joint numerical range 
of matrices of order $3$ having two elliptic faces (an affine transformation of 
\cite[Example 6.3]{Szymanski-etal2018}).
\end{ourexample}
\begin{ourexample}[8]\label{ex-8}
Three droplet faces: $a_2=3$ and $a_0=a_1=a_3=0$\\ 
Exemplary matrices:
\[ 
\begin{bmatrix}
 2 & 0 & 0 & 0\\
 0 & 2 & 0 & 0\\
 0 & 0 & 1 & \ii\\
 0 & 0 & -\ii & 1
\end{bmatrix}, 
\begin{bmatrix}
 1 & 0 & -\ii & 0\\
 0 & 2 & 0 & 0\\
 \ii & 0 & 1 & 0\\
 0 & 0 & 0 & 2
\end{bmatrix},
\begin{bmatrix}
 1 & 0 & 0 & \ii\\
 0 & 2 & 0 & 0\\
 0 & 0 & 2 & 0\\
 -\ii & 0 & 0 & 1
\end{bmatrix} 
\]
Rank-1 tuples: $(-2,1,0,0)$, $(-2,0,1,0)$, and $(-2,0,0,1)$
\end{ourexample}
\begin{ourexample}[9]\label{ex-9}
A droplet and a triangular face: $a_2=a_3=1$ and $a_0=a_1=0$\\
Exemplary matrices:
\[
\begin{bmatrix}
 1 & 0 & 0 & 0\\
 0 & 0 & 1 & 0\\
 0 & 1 & 0 & 0\\ 
 0 & 0 & 0 & 1
\end{bmatrix}, 
\begin{bmatrix}
 0 & 0 & \sqrt2 & 0\\ 
 0 & 0 & 0 & 0\\ 
 \sqrt2 & 0 & 0 & 0\\ 
 0 & 0 & 0 & 1
\end{bmatrix}, 
\,\diag(0,0,-2,0)
\]
Rank-1 tuples: $(0,0,0,1)$  and  $(-1,1,0,0)$
\par
This example was obtained by adding a corner point to the joint numerical range 
of three hermitian $3\times 3$ matrices having one ellipse and one segment on its 
boundary (cf.\ \cite[Example~6.7]{Szymanski-etal2018}).
The corner point in this class is inevitable, see Theorem~\ref{thm:3segments}.
\end{ourexample}
\begin{ourexample}[10]\label{ex-10}
A droplet and two triangular faces: $a_2=1, a_3=2$ and $a_0=a_1=0$\\
Exemplary matrices:
\[
\begin{bmatrix}
 1 & -\ii & 0 & 0\\
 \ii & 1 & 0 & 0\\
 0 & 0 & 0 & 0\\
 0 & 0 & 0 & 0
\end{bmatrix}, 
\begin{bmatrix}
 0 & 1 & 0 & 0\\
 1 & 0 & 0 & 0\\
 0 & 0 & 1 & 0\\ 
 0 & 0 & 0 & 1
\end{bmatrix},
\,\diag(0,0,1,0)
\]
Rank-1 tuples: $(0,0,0,1)$, $(-1,0,1,0)$, and $(0,1,0,0)$
\par
The triple of matrices is unitarily reducible and hence the joint numerical range 
is the convex hull of the points $(0,1,1)$, $(0,1,0)$, and the circle in the 
$x$-$y$-plane with center $(1,0,0)$ and radius $1$. The configuration of faces 
can be recovered using Lemma~\ref{lem:faces-convex-hull}.
\par
Note that due to Theorem~\ref{thm:3segments}, the example of this
class necessarily has two corner points (endpoints of the common edge of
the triangular faces), hence it is a convex hull
of a segment and an ellipsoid (the joint numerical range of
$2\times 2$ matrices). Unless the latter
is degenerated to an ellipse coplanar with exactly one of the corners,
the droplet face cannot appear (cf.\ Lemma~\ref{lem:faces-convex-hull}).
\end{ourexample}
\begin{ourexample}[11]\label{ex-11}
Two droplets and a triangle: $a_2=2$, $a_3=1$, and $a_0=a_1=0$\\
Exemplary matrices:
\[
\begin{bmatrix}
 1 & 0 & \ii\sqrt2 & 0\\
 0 & 0 & 0 & 0\\
 -\ii\sqrt2 & 0 & 2 & 0\\ 
 0 & 0 & 0 & 0
\end{bmatrix},
\begin{bmatrix}
 0 & 0 & 0 & 0\\
 0 & 1 & -\ii\sqrt2 & 0\\
 0 & \ii\sqrt2 & 2 & 0\\ 
 0 & 0 & 0 & 0
\end{bmatrix},
\,\diag(0,0,2,0)
\]
Rank-1 tuples: $(0,1,0,0)$, $(0,0,1,0)$, and $(0,0,0,1)$
\par
An example of this type can be obtained by adding a corner point to
the joint numerical range of $3\times 3$ matrices having two
ellipses and a segment in the boundary (see \cite[Example~6.3]{Szymanski-etal2018}). According to Lemma~\ref{lem:int3non-elliptic}, there has to be a corner point and all 
examples in the class share this construction (cf.\ also Lemma~\ref{lem:faces-convex-hull}). 
\end{ourexample}
\begin{ourexample}[12]\label{ex-12}
Single triangular face: $a_3=1$ and $a_0=a_1=a_2=0$\\
{Exemplary matrices:
\[ 
\begin{bmatrix}
 \sqrt3 & 0 & 0 & -1\\
 0 & -\frac{\sqrt3}2 & 0 & 1\\
 0 & 0 & -\frac{\sqrt3}2& -\ii\\
 -1 & 1 & \ii & 0
\end{bmatrix},
\begin{bmatrix}
 0 & 0 & 0 & \ii\\
 0 & \frac32 & 0 & -\ii\\
 0 & 0 & -\frac32 & 1\\
 -\ii & \ii & 1 & 0
\end{bmatrix},
\,\diag(0,0,0,-2)
\]
Rank-1 tuple: $(0,0,0,1)$} 
\par
This example is constructed similarly to Examples~\hyperref[ex-1]{E1}, \hyperref[ex-2]{E2} 
and~\hyperref[ex-6]{E6}.
\end{ourexample}
\begin{ourexample}[13]\label{ex-13}
Two triangular faces: $a_3=2$ and $a_0=a_1=a_2=0$\\
Exemplary matrices:
\[
\diag(0,0,1,-1),
\begin{bmatrix}
 1 & 0 & 0 & 0\\
 0 & 1 & 0 & 0\\
 0 & 0 & 0 & \frac12\\
 0 & 0 & \frac12 & 0
\end{bmatrix}, 
\begin{bmatrix}
 0 & 1 & 0 & 0\\
 1 & 0 & 0 & 0\\
 0 & 0 & 0 & -\ii\\
 0 & 0 & \ii & 0
\end{bmatrix} 
\]
Rank-1 tuples: $(-2,\sqrt{3},2,0)$  and  $(-2,-\sqrt{3},2,0)$
\par
By Theorem~\ref{thm:3segments}, the joint numerical range in this
class has two corner points and therefore the matrices admit
a common two-dimensional reducing subspace. The joint numerical range
is necessarily the convex hull of two points and a, possibly degenerated,
ellipsoid $E$, similarly as in Example~\hyperref[ex-10]{E10}.
The ellipsoid $E$ cannot intersect the segment between the corners and cannot
be enclosed in a plane containing one of them (the latter would
lead to class no.~10 or no.~14).
\end{ourexample}
\begin{ourexample}[14]\label{ex-14}
Four triangular faces: $a_3=4$ and $a_0=a_1=a_2=0$\\
Exemplary matrices:
\[
\diag(1,1,-1,-1),
\,\diag(1,-1,1,-1),
\,\diag(1,-1,-1,1)
\]
The case $a_3=4$ is attained solely in case of the joint numerical range
being a tetrahedron. For every tetrahedral joint numerical range, there 
exists an orthonormal basis simultaneously diagonalizing all three matrices.
\end{ourexample}
%
%
%
\section{One-dimensional faces and corner points}\label{sec:onedim}
Let $A_1,A_2,A_3$ be hermitian $n\times n$ matrices and let $W=W(A_1,A_2,A_3)$ have
dimension $\dim(W)=3$. We prove that the intersection of any three mutually distinct 
one-dimensional faces of $W$ that have a point in common is a corner point of $W$ if 
$n=4$. A counterexample is given for $n=5$.
\par
\begin{remark}[Recurrent arguments]\label{rem:recurrent-arg}
\begin{enumerate}[\ (a)]
\item
Every one-dimensional nonexposed face $G_1$ of $W$ is an exposed face of a 
two-dimensional exposed face $H_1$ of $W$, see the discussion below
equation~\eqref{eq:access-F}. The face $H_1$ is unique (because a 
one-dimensional face that is included in two distinct two-dimensional 
exposed faces is exposed). If $n=4$, then $H_1$ is an example of what we call a 
non-elliptic face in Definition~\ref{def:non-elliptic}.
\item
If two or more mutually distinct one-dimensional faces of $W$ have a point in 
common, then their intersection is an extreme point of $W$ (because the infimum 
in the lattice of faces is the intersection).
\end{enumerate}    
\end{remark}
\begin{lemma}\label{lem:3exp-segments}
Let three mutually distinct one-dimensional faces of $W$\!, at least two of which 
are exposed, have a point in common. Then their intersection is a corner point of 
$W$.
\end{lemma}
\begin{proof}
Let $G_1,G_2,G_3$ denote the one-dimensional faces in question and assume 
that $G_2,G_3$ are exposed. Then $G_2\cap G_3=\{p\}$ holds for
an exposed point $p$ of $W$, as the infimum in the lattice of exposed 
faces is the intersection. Also, $G_1\cap G_2\cap G_3=\{p\}$.
\par
If $G_1$ is exposed then, by Proposition~\ref{pro:normal-cones}, the normal 
cones $N_W(G_1)$, $N_W(G_2)$, $N_W(G_3)$ are rays or two-dimensional convex 
cones that are mutually distinct and that are proper faces of $N_W(p)$. This 
implies $\dim N_W(p)=3$. Let $G_1$ be a nonexposed face and $H_1$ the 
unique two-dimensional face containing $G_1$. The same argument as before
when $N_W(G_1)$ is replaced with $N_W(H_1)$ proves $\dim N_W(p)=3$.
\end{proof}
Example~\ref{ex:3segments} shows that the conclusion of Lemma~\ref{lem:3exp-segments} 
is no longer valid for $5\times 5$ matrices under the weaker assumption that only
one of the one-dimensional faces is exposed. Therefore, Theorem~\ref{thm:3segments} is 
a special result on $4\times 4$ matrices, not applicable to matrices of higher order.
\par
\begin{example}\label{ex:3segments}
Consider the convex set $C$ in $\R^3$ that is the convex hull of the unit circle 
in the $y$-$z$-plane and the triangle with vertices $p_1:=(-1,-1,0)\tp$, 
$p_2:=(-1,1,0)\tp$, $p_3:=(1,0,0)\tp$ in the $x$-$y$-plane.
\par
The segments $[p_1,q]$ and $[p_2,q]$ are nonexposed faces of $C$ and the segment 
$[p_3,q]$ is an exposed face of $C$. All three segments lie in the boundary of $C$ 
and intersect at the point $q:=(0,0,1)\tp$. But, $q$ is not a corner point of $C$. 
Note that $C$ can be realized as the joint numerical range of the hermitian 
$5\times 5$ matrices 
\[
\begin{bmatrix}
0 & 0\\
0 & 0
\end{bmatrix}
\oplus
\diag(-1,-1,1),
\quad
\begin{bmatrix}
0 & 1\\
1 & 0
\end{bmatrix}
\oplus
\diag(-1,1,0),
\quad
\begin{bmatrix}
0 & -\ii\\
\ii & 0
\end{bmatrix}
\oplus
\diag(0,0,0),
\]
which (by Lemma~\ref{lem:convWs} and Remark~\ref{rem:faces-JNR}~(b))
is the convex hull of the unit disk 
\[\textstyle
W\left(
\begin{bmatrix}
0 & 0\\
0 & 0
\end{bmatrix},
\begin{bmatrix}
0 & 1\\
1 & 0
\end{bmatrix},
\begin{bmatrix}
0 & -\ii\\
\ii & 0
\end{bmatrix}\right)
\]
in the $y$-$z$-plane and the points $p_1,p_2,p_3$ in the $x$-$y$-plane.
\end{example}
\begin{figure}[hbt]\centering
\includegraphics[width=4.5cm]{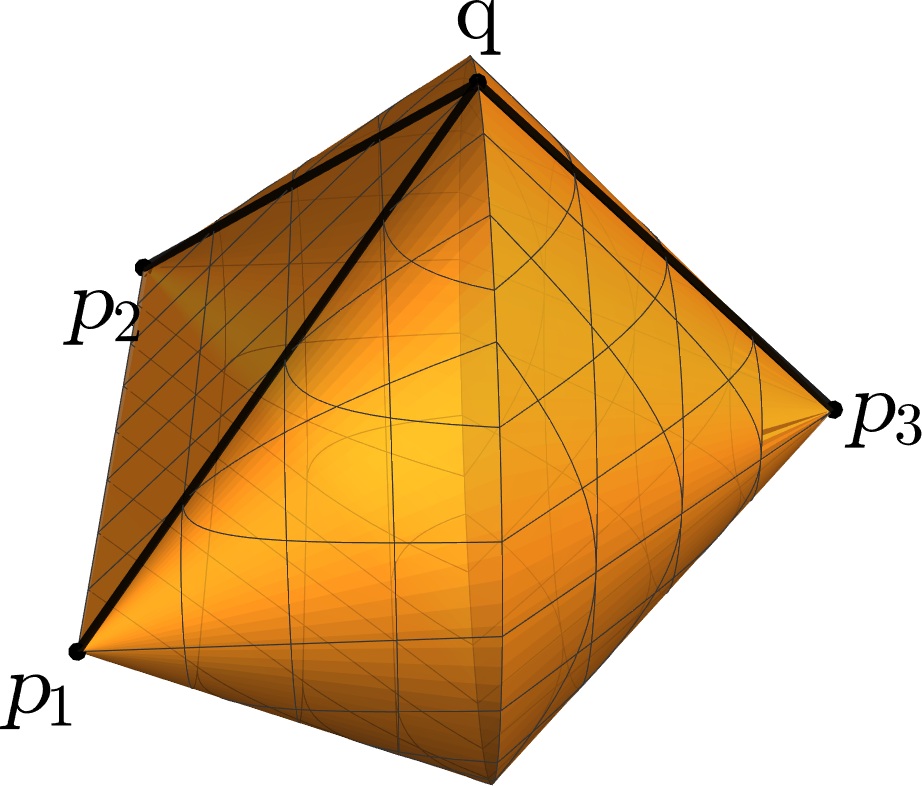}
\caption{%
The joint numerical range $C$ described in Example~\ref{ex:3segments}. The point $q$ 
is not a corner point of $C$ although it is the intersection of three mutually distinct 
one-dimensional faces of $C$. This can only happen for matrices of order five or greater
by Theorem~\ref{thm:3segments}.}
\label{fig:ex5-3}
\end{figure}
From here on, the hermitian matrices $A_1,A_2,A_3$ are assumed to have dimension $n=4$. 
In the following Lemma~\ref{lem:matrices2segements}, we simplify the matrices 
$A_1,A_2,A_3$ using the access sequence technique of Theorem~\ref{thm:PF-iso} 
and we customize the notation.
\par
\begin{lemma}\label{lem:matrices2segements}
Suppose that $W$ has a non-elliptic face $F_1$. Suppose further that $F_1$ has 
a one-dimensional face $F_2$ that intersects a one-dimensional exposed face $G$ 
of $W$ and $G\neq F_2$. Then $F_2\cap G=\{p\}$ is a singleton, which
is an extreme point of $W$. Up to a unitary similarity and an invertible affine 
transformation, the access sequence of faces 
$W\supset F_1\supset F_2\supset\{p\}$ corresponds to the access sequence 
$\id_4\geq P_1\geq P_2 \geq P_3$ of orthogonal projections 
\[
P_1 =\diag(1,1,1,0),
\quad
P_2 =\diag(1,1,0,0),
\quad
P_3 =\diag(1,0,0,0),
\]
and 
\begin{equation}\label{eq:custMat}
A_1=\diag(0,0,0,1),
\quad
A_2=0\oplus
C,
\quad
A_3=
\begin{bmatrix}
\diag(0,0,1) & \ket{\varphi}\\
\bra{\varphi} & 0
\end{bmatrix},
\end{equation}
where $C\in\M_3$ is positive semidefinite of rank two and 
$\ket{\varphi}=(w_1,w_2,w_3)\tp\in\C^3$. Note $p=(0,0,0)\tp$. The 
projection associated to $G$ is the kernel projection of $A_2$.
\end{lemma}
\begin{proof}
Since $F_2\cap G=\{p\}$ is an extreme point of $W$, it follows that 
$p$ is an endpoint of the segment $F_2$ (and of $G$). We can therefore 
use the access sequences provided in Example~\ref{exa:rankF}. The projection 
$Q$ associated to $G$ has rank $\rk(Q)=2$. Otherwise, $\rk(Q)=3$ would 
contradict Lemma~\ref{lem:int2non-elliptic} because $G\neq F_1$.
\par
We analyze the exposed faces $F_1$ and $G$ using access sequences of lengths one. 
Applying an invertible affinity as described by the equations \eqref{eq:affT} and 
\eqref{eq:affA}, we can take $A_1$ (resp., $A_2$) such that $P_1$ 
(resp., $Q$) is the spectral projection of $A_1$ (resp., $A_2$) corresponding to 
the smallest eigenvalue. Thus, we take $A_1=\diag(0,0,0,1)$ and $A_2$ a positive 
semidefinite matrix of rank two. Then $A_2=0\oplus C$ holds for a positive 
semidefinite $3\times 3$ matrix $C$, because $P_3\leq Q$ is implied by 
$p\in G$. We write
\[
C=\begin{bmatrix}
B & \ket{\eta}\\
\bra{\eta} & a
\end{bmatrix},
\quad
A_2=0\oplus
\begin{bmatrix}
B & \ket{\eta}\\
\bra{\eta} & a
\end{bmatrix},
\quad
\text{and}
\quad
A_3=
\begin{bmatrix}
D & \ket{\xi}\\
\bra{\xi} & 0
\end{bmatrix},
\]
where $B\in\M_2$ is positive semidefinite, $\ket{\eta}\in\C^2$, $a\geq 0$, 
$D\in\M_3$ is hermitian, and $\ket{\xi}\in\C^3$. The bottom right entry of 
$A_3$ was tacitly erased adding a scalar multiple of $A_1$ to $A_3$.
\par
Considering the access sequence $\id_4\geq P_1\geq P_2$ of length two, we 
know that $P_2$ is the spectral projection of a linear combination of 
matrices $P_1A_iP_1$, $i=1,2,3$, corresponding to the smallest spectral 
value in the algebra $\M_3\oplus 0$. There exist
$\lambda_1,\lambda_2,\lambda_3\in\R$ such that
\begin{align}\label{eq:P2acc}
\nonumber
\diag(0,0,1)\oplus 0
&=P_1(\lambda_1\id_4+\lambda_2 A_2+\lambda_3 A_3)P_1\\
&=(\lambda_1\id_3+\lambda_2(0\oplus B)+\lambda_3 D)\oplus 0.
\end{align}
A solution to equation \eqref{eq:P2acc} exists by Theorem~\ref{thm:PF-iso}. We 
use it to solve for $D$. First, note that $\lambda_1=\lambda_3=0$ is not 
possible, because $B$ is not a scalar multiple of $\diag(0,1)$. If it were, 
then $C\geq 0$ would imply $\ket{\eta}=(0,z)\tp$ for some $z\in\C$, see for 
example \cite{HornZhang2005}. Then $P_2=Q$ would be the kernel projection of 
$A_2$ in contradiction to the assumption $F_2\neq G$. Second, $\lambda_1\neq\lambda_3=0$ 
is not possible, as the top left entries of the two sides of equation \eqref{eq:P2acc} 
would differ. This proves $\lambda_3\neq 0$ and shows that there are $s_1,s_2,s_3\in\R$ 
such that $s_3\neq 0$ and 
\[
D=s_1\id_3+s_2(0\oplus B)+s_3\diag(0,0,1).
\]
Applying another affinity, we replace $A_3$ with 
$A_3-s_1\id_4+(s_1+as_2)A_1-s_2A_2$ divided by $s_3$. This yields the matrices 
above.
\end{proof}
The following lemma deserves to be stated explicitly as it is used 
twice, in Propositions~\ref{pro:2segments} and ~\ref{pro:corner2non-ell}.
\par
\begin{lemma}\label{lem:3x3nonexp1}
Assume $A_1,A_2,A_3\in\C P\oplus P'\M_4P'$, where $P\in\cP_4$ is a nonzero 
orthogonal projection and $P':=\id_4-P$. Let $p_i\in\R$ be such that 
$A_iP=p_iP$, $i=1,2,3$. If $W$ has a one-dimensional nonexposed face, then 
$(p_1,p_2,p_3)\tp$ is a corner point of $W$. 
\end{lemma}
\begin{proof}
Let $s<4$ be the rank of the projection $P'=\id_4-P$. By Lemma~\ref{lem:convWs}, 
the joint numerical range $W$ is the convex hull of $p$ and the joint numerical 
range 
\[
W':=W_{P'\M_nP'}(P'A_1P',P'A_2P',P'A_3P'),
\] 
which by Remark~\ref{rem:faces-JNR}~(a) and~(b) is the joint numerical range of three 
hermitian $s\times s$ matrices. If $p$ lies outside $W'$, then $p$ is a 
corner point of $W$. Otherwise, $p\in W'$ leads to a contradiction, because
$W'=W$ has no one-dimensional nonexposed faces. Indeed, if $s=3$, then every 
nonexposed face of $W'$ is a nonexposed point by \cite[Lemma~4.3]{Szymanski-etal2018}. The convex set $W'$ has no nonexposed faces at 
all if $s=2$ (where $W'$ is a linear image of a three-dimensional Euclidean 
ball) or if $s=1$ (where $W'$ is a singleton).
\end{proof}
The following proposition highlights a special property of 
$4\times 4$ matrices, whose analogue fails for the $5\times 5$ matrices in 
Example~\ref{ex:3segments} above.
\par
\begin{proposition}\label{pro:2segments}
Suppose a one-dimensional nonexposed face $F_2$ of $W$ intersects a 
one-di\-men\-sion\-al exposed face $G$ of $W$. Let $p$ be the extreme point 
of $W$ such that $F_2\cap G=\{p\}$. Then $p$ is a corner point of $W$ or $G$ 
is contained in the unique non-elliptic face of $W$ that includes $F_2$.
\end{proposition}
\begin{proof}
The face $F_2$ is an exposed face of a non-elliptic face $F_1$. Thus, the 
assumptions of Lemma~\ref{lem:matrices2segements} are fulfilled and we use 
its results and notation.
\par
Zooming in onto the exposed face $G$, we use a unitary $4\times 4$ matrix $U$
such that the projection $Q$ associated to $G$ is unitarily similar to 
$\id_2\oplus 0=UQU^\ast$. Then by Remark~\ref{rem:faces-JNR}~(b), the hermitian 
$2\times 2$ matrices $B_i:=(UA_iU^\ast)[2]$, $i=1,2,3$, are such that
\[
G=W(B_1,B_2,B_3).
\]
As $Q$ is the kernel projection of $A_2$, the first two columns of $U^\ast$ may 
be taken as the unit vectors $(1,0,0,0)\tp$ and $(0,z_1,z_2,z_3)\tp$. Hence
\[
B_1=\begin{bmatrix}
0 & 0\\
0 & |z_3|^2
\end{bmatrix}
\quad
\text{and}
\quad
B_3=\begin{bmatrix}
0 & z_3 w_1\\
\bar z_3\bar w_1 & \ast
\end{bmatrix}.
\]
If $z_3=0$, then $Q\leq P_1$ implies $G\subset F_1$ as desired. Otherwise, 
if $z_3\neq 0$, then we invoke the fact that $G$ is a segment. This implies that 
$B_1$ and $B_3$ commute. Then $w_1=0$ follows and it implies 
$A_1,A_2,A_3\in 0\oplus\M_3$. Now Lemma~\ref{lem:3x3nonexp1} shows that 
$p=(0,0,0)\tp$ is a corner point of $W$.
\end{proof}
\begin{remark}\label{rem:two1D}
The optimality of Proposition~\ref{pro:2segments} can be seen from examples.
\par
\begin{enumerate}[(a)]
\item
The first disjunct of the disjunctive assertion of Proposition~\ref{pro:2segments} 
fails for the joint numerical range $W$ in Example~\hyperref[ex-7a]{E7a}. This
numerical range has two faces $H_1,H_2$ of droplet shapes. The segment $G_0:=H_1\cap H_2$ 
is an exposed face of $W$ and the other boundary segment $G_1$ of $H_1$ is a nonexposed 
face of $W$. As $G_0$ and $G_1$ intersect at a nonexposed points, this point cannot be a 
corner point.
\item
The second disjunct fails for certain joint numerical ranges $W$ having one face 
of droplet shape. Let $W'$ be the three-dimensional joint numerical range of three 
hermitian $3\times 3$ matrices, such that $W'$ has exactly one face of elliptic 
shape, say $E$, but $W'$ is not a pyramid based on $E$, see 
\cite[Example~6.2]{Szymanski-etal2018}. Let $W$ be the convex hull of $W'$ and a 
point $p$ in the affine hull of $E$ but outside $W'$. The convex hull $H$ of $p$ 
and $E$ is a face of $W$ of droplet shape, whose two boundary segments $G_1,G_2$ 
are nonexposed faces of $W$. Still, $W$ has a whole family of one-dimensional 
exposed faces containing $p$ that are not included in $H$.
\item 
The assertion of Proposition~\ref{pro:2segments} cannot be strengthened to 
an exclusive disjunction. This can be seen from Example~\hyperref[ex-7b]{E7b}
(or \hyperref[ex-9]{E9}).
This numerical range has two faces $H_1,H_2$ of droplet shapes. The segment 
$G_0:=H_1\cap H_2$ is an exposed face of $W$. The boundary segment $G_i$ of $H_i$ 
that differs from $G_0$ is a nonexposed face of $W$, $i=1,2$, and the 
segments $G_0,G_1,G_2$ intersect at a point $p$. Theorem~\ref{thm:3segments} 
shows that $p$ is a corner point of $W$, thereby proving the first disjunct. 
As $G_0$ is included in the unique non-elliptic face $H_1$ that includes $G_1$,
the second disjunct is also true.
\item 
Matrices of higher order (i.e.\ $n\geq 5$) are excluded from the statement of 
Proposition~\ref{pro:2segments} by Example~\ref{ex:3segments}.
\end{enumerate}
\end{remark}
\begin{proposition}\label{pro:corner2non-ell}
Suppose $G_1,G_2$ are distinct one-dimensional nonexposed faces of $W$ that are 
included in distinct non-elliptic faces of $W$, and let $G$ be a one-dimensional 
exposed face of $W$. Suppose $G_1\cap G_2\cap G\neq\emptyset$. Then 
$G_1\cap G_2\cap G=\{p\}$ is a corner point of $W$.
\end{proposition}
\begin{proof}
Let $H_i$ denote the non-elliptic face including $G_i$, $i=1,2$. If $G$ is not 
included in $H_1$, then $p$ is a corner point of $W$ by Proposition~\ref{pro:2segments}.
\par
Let $G\subset H_1$. We use results and notation from 
Lemma~\ref{lem:matrices2segements} assuming $F_1:=H_1$ and $F_2:=G_1$. The 
projection associated to $H_1$ is the kernel projection $P_1$ of 
$A_1=\diag(0,0,0,1)$. The projection associated to $G$ is the kernel 
projection $Q$ of $A_2$, which is a positive semidefinite matrix of rank two, 
and which simplifies (because of $G\subset H_1$) to 
\[
A_2=0\oplus
\begin{bmatrix}
\ketbra{\psi} & z\ket{\psi}\\
\bar z\bra{\psi} & a
\end{bmatrix},
\]
where $\ket{\psi}\in\C^2$ is a unit vector, $a>0$, and $z\in\C$ satisfies 
$|z|<\sqrt{a}$.
\par
One easily finds that the projection $R$ associated to $H_2$ is the kernel projection of 
\[
A:=A_2-\lambda A_1
=0\oplus
\begin{bmatrix}
\ketbra{\psi} & z\ket{\psi}\\
\bar z\bra{\psi} & |z|^2
\end{bmatrix},
\]
where $\lambda=a-|z|^2$ is chosen so $A_2-\lambda A_1$ becomes a positive semidefinite matrix 
of rank one. 
\par
Similarly to Proposition~\ref{pro:2segments}, we use a unitary $4\times 4$ matrix $U$
such that $R$ is unitarily similar to $\id_3\oplus 0=URU^\ast$. Then the 
hermitian $3\times 3$ matrices $B_i:=(UA_iU^\ast)[3]$, $i=1,2,3$, yield
\[
H_2=W(B_1,B_2,B_3).
\]
Taking the first three columns of $U^\ast$ as the orthonormal vectors
\begin{align*}
\ket{1}&:=(1,0,0,0)\tp,\\
\ket{2}&:=0\oplus\ket{\psi^\perp}\oplus 0,\\
\ket{3}&:=(0\oplus (-z)\ket{\psi}\oplus1)/\sqrt{1+|z|^2}, 
\end{align*}
where $\ket{\psi^\perp}\in\C^2$ is a unit vector orthogonal to $\ket{\psi}$, 
one has
\[
B_1=
\begin{bmatrix}
0 & 0 & 0\\
0 & \ast & \ast\\
0 & \ast & \ast
\end{bmatrix},
\quad
B_2=
\begin{bmatrix}
0 & 0 & 0\\
0 & \ast & \ast\\
0 & \ast & \ast
\end{bmatrix},
\quad
B_3=\frac{1}{\sqrt{1+|z|^2}}
\begin{bmatrix}
0 & 0 & w_1\\
0 & \ast & \ast\\
w_1 & \ast & \ast
\end{bmatrix}.
\]
Finally, the assumption that $G$ and $G_2$ intersect at $p=(0,0,0)\tp$ 
takes effect. It shows that $p$ is a corner point of $H_2$. Therefore, 
$B_1,B_2,B_3\in 0\oplus\M_2$ holds by Proposition~\ref{pro:corner-points}, 
which implies $w_1=0$, hence $A_1,A_2,A_3\in 0\oplus\M_3$. Now 
Lemma~\ref{lem:3x3nonexp1} shows that $p$ is a corner point of $W$.
\end{proof}
\begin{lemma}\label{lem:3nonexp-segments}
The intersection of any three mutually distinct one-dimensional 
nonexposed faces of $W$ is empty.
\end{lemma}
\begin{proof}
Let $G_1,G_2,G_3$ be mutually distinct one-dimensional nonexposed faces of 
$W$ and let $H_i$ denote the non-elliptic face of $W$ that includes $G_i$,
$i=1,2,3$. We distinguish three cases. 
\par
First, if $H_1=H_2=H_3$ then $H_1$ is a triangle and the intersection 
$G_1\cap G_2\cap G_3$ of its three sides is empty. 
\par
Second, if exactly two of $H_1,H_2,H_3$ are equal, say $H_1\neq H_2=H_3$,
then  by Lemma~\ref{lem:int2non-elliptic}, $S:=H_1\cap H_2$ is 
a one-dimensional face of $H_1$ and of $H_2$, and an exposed face of $W$. 
As the $G_i$'s are faces of $H_1$ and of $H_2$, too, the number of 
one-dimensional faces of $H_i$, summed over $i=1,2$, is at least five 
(two for $S$ counted twice plus three for the $G_i$'s). This forces one 
of the $H_i$'s, say $H_1$, to be a triangle and the other, 
say $H_2$, to be a droplet or a triangle. In either case, see the following 
figure, we have $G_1\cap G_2\cap G_3=\emptyset$.

{\centering
\begin{tikzpicture}
 \coordinate (1) at (-1.732,1); 
 \coordinate (2) at (0,0); 
 \coordinate (3) at (1.732,1); 
 \coordinate (4) at (0,2); 
 \path[draw,thick,dashed] (4) -- (1) -- (2) -- (3);
 \draw[thick,dashed] ([shift=(-60:1.155cm)]1.155,2) arc (-60:180:1.155cm);
 \path[draw,thick] (2) -- (4);
 \node at ([shift=({-.6,0})]$(2)!0.5!(4)$) {$H_1$};
 \node at (1.155,2) {$H_2$};
\end{tikzpicture}
\qquad\qquad
\begin{tikzpicture}
 \coordinate (1) at (-1.732,1); 
 \coordinate (2) at (0,0); 
 \coordinate (3) at (1.732,1); 
 \coordinate (4) at (0,2); 
 \path[draw,thick,dashed] (1) -- (2) -- (3) -- (4) -- cycle;
 \path[draw,thick] (2) -- (4);
 \node at ([shift=({-.6,0})]$(2)!0.5!(4)$) {$H_1$};
 \node at ([shift=({.6,0})]$(2)!0.5!(4)$) {$H_2$};
\end{tikzpicture}\par}
\par
Third, let $H_1,H_2,H_3$ be mutually distinct. Then $S_1:=H_2\cap H_3$, 
$S_2:=H_3\cap H_1$, and $S_3:=H_1\cap H_2$ are mutually distinct 
one-dimensional exposed face of $W$ by Lemma~\ref{lem:int2non-elliptic} 
and Lemma~\ref{lem:int3non-elliptic}. The numbers of one-dimensional 
faces of $H_i$, summed over $i=1,2,3$, yield at least nine (six for the 
$S_i$'s counted twice plus three for the $G_i$'s). This proves that 
each of the $H_i$'s is a triangle. Now Lemma~\ref{lem:not0003} shows 
that $W$ is a tetrahedron, which has no nonexposed faces, a contradiction.
\end{proof}
\begin{theorem}\label{thm:3segments}
If the intersection of three mutually distinct one-dimensional faces of $W$ is nonempty, 
then this intersection is a corner point of $W$.
\end{theorem}
\begin{proof}
Let $G_1,G_2,G_3$ denote the one-dimensional faces in question. Then $G_1\cap G_2\cap G_3=\{p\}$ 
where $p$ is an extreme point of $W$ (this is the statement of 
Remark~\ref{rem:recurrent-arg}~(b)).
\par
By Lemma~\ref{lem:3exp-segments} and Lemma~\ref{lem:3nonexp-segments},
the claim is true unless exactly two of the faces $G_1,G_2,G_3$ are
nonexposed.
\par
Let $G_1,G_2$ be one-dimensional nonexposed faces, let $H_i$ be the 
non-elliptic face including $G_i$, $i=1,2$, and let $G_3$ be a 
one-dimensional exposed face of $W$. First, let $H_1=H_2$. If $G_3\subset H_1$,
then $H_1$ is a triangle and the intersection $G_1\cap G_2\cap G_3$ of its 
three sides is empty. Since $G_3\subset H_1$ fails, Proposition~\ref{pro:2segments} 
shows that $p$ is a corner point of $W$. Second, if $H_1\neq H_2$ then $p$ is 
a corner point of $W$ by Proposition~\ref{pro:corner2non-ell}.
\end{proof}
%
%
\section{On elliptic faces}
\label{sec:elliptic-faces}
Let $A_1,A_2,A_3$ be hermitian $4\times 4$ matrices. We collect observations 
on elliptic faces of the joint numerical range $W=W(A_1,A_2,A_3)$. We assume 
$\dim(W)=3$.
\par
\begin{lemma}\label{lem:int-non-elliptic-elliptic}
If an elliptic face of $W$ differs from a rank-3 face of $W$\!, then their 
intersection is an exposed point of $W$\!, and hence an exposed point of that rank-3 face.
\end{lemma}
\begin{proof}
Let $F$ be a face of $W$ that is an ellipse and let $G$ be a rank-3 face. Let $P$ (resp., $Q$) 
denote the projection associated to $F$ (resp., $G$) in Theorem~\ref{thm:PF-iso}. Then $P$ has 
rank two or three and $Q$ has rank three. Hence the intersection of the ranges of $P$ and $Q$
is at least one-dimensional, and consequently the faces $F$ and $G$ have a 
non-empty intersection.
\par
Since $F$ and $G$ have dimension two, they are exposed faces of $W$ and so is their 
intersection. As $F$ is an ellipse, this intersection is a singleton, hence an 
exposed point of $W$. \emph{A fortiori}, this point is an exposed point of $G$.
\end{proof}
\begin{remark}
In view of Lemma~\ref{lem:int-non-elliptic-elliptic}, one may ask whether an 
elliptic face of $W$ can intersect the endpoints of a one-dimensional face of 
a non-elliptic face. This is impossible for a type-$1$ non-elliptic face (loaf), 
as the endpoints of its one-dimensional face are nonexposed points.

An elliptic face of $W$ can intersect the vertex of a non-elliptic droplet face.
An example is the joint numerical range (see Figure~\ref{fig:rem6.2}~(a)) of the matrices 
\[
\begin{bmatrix}
 0 &    1\\   1 &  0
\end{bmatrix}
\oplus
\begin{bmatrix}
 1 &    0\\   0 &  1    
\end{bmatrix},
\begin{bmatrix}
 0 & -\ii\\ \ii &  0   
\end{bmatrix}
\oplus
\begin{bmatrix}
 0 &    1\\   1 &  0  
\end{bmatrix},
\begin{bmatrix}
-2 &    0\\   0 & -2    
\end{bmatrix}
\oplus
\begin{bmatrix}
 0 & -\ii\\ \ii &  0   
\end{bmatrix},
\]
which is the convex hull of the unit circle in the $x$-$y$-plane 
translated by $-2$ in $z$-direction and the unit circle in the
$y$-$z$-plane translated by $1$ in $x$-direction.
\par

\noindent
\quad An elliptic face of $W$ can also intersect the vertex of a triangular
face. An example is the joint numerical range (see Figure~\ref{fig:rem6.2}~(b)) of the matrices
\[
\begin{bmatrix}
 0 &    1\\   1 &  0   
\end{bmatrix}
\oplus
\begin{bmatrix}
 1 &    0\\   0 & -1   
\end{bmatrix},
\begin{bmatrix}
 0 & -\ii\\ \ii &  0   
\end{bmatrix}
\oplus
\begin{bmatrix}
 0 &    0\\   0 &  0
\end{bmatrix},
\quad
\begin{bmatrix}
 0 &    0\\   0 &  0   
\end{bmatrix}
\oplus
\begin{bmatrix}
 1 &    0\\   0 &  1   
\end{bmatrix},
\]
which is the convex hull of the unit circle in the $x$-$y$-plane 
and the points $(1,0,1)\tp$, $(-1,0,1)\tp$.

\end{remark}
\begin{figure}[hbt]\centering
(a) \includegraphics[width=3.2cm]{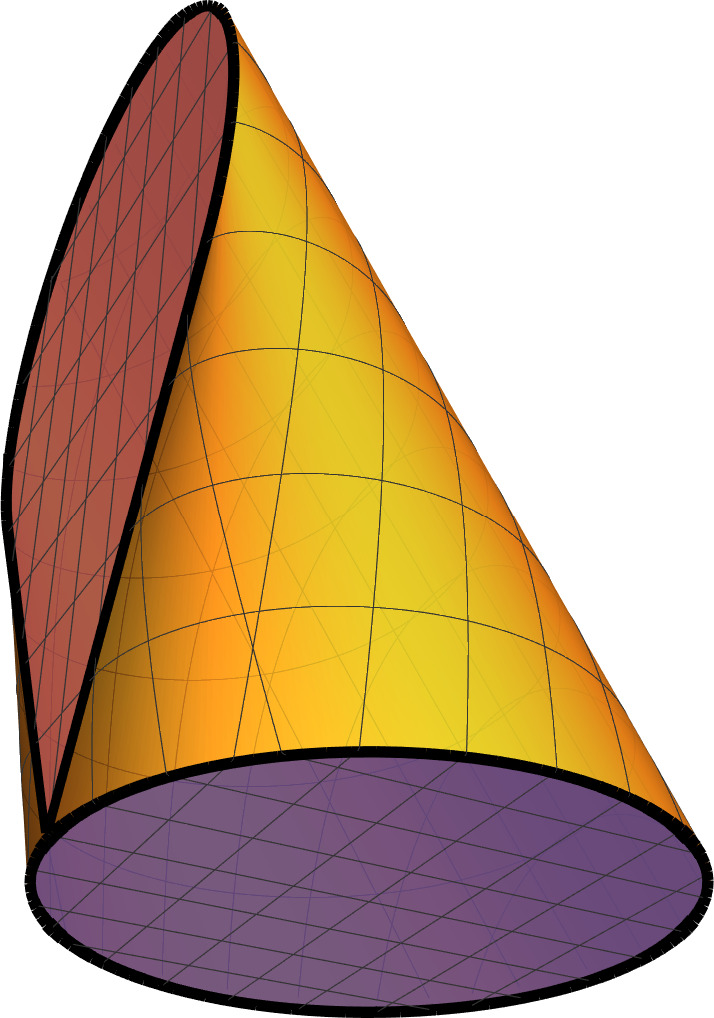}\qquad
(b) \includegraphics[width=3.2cm]{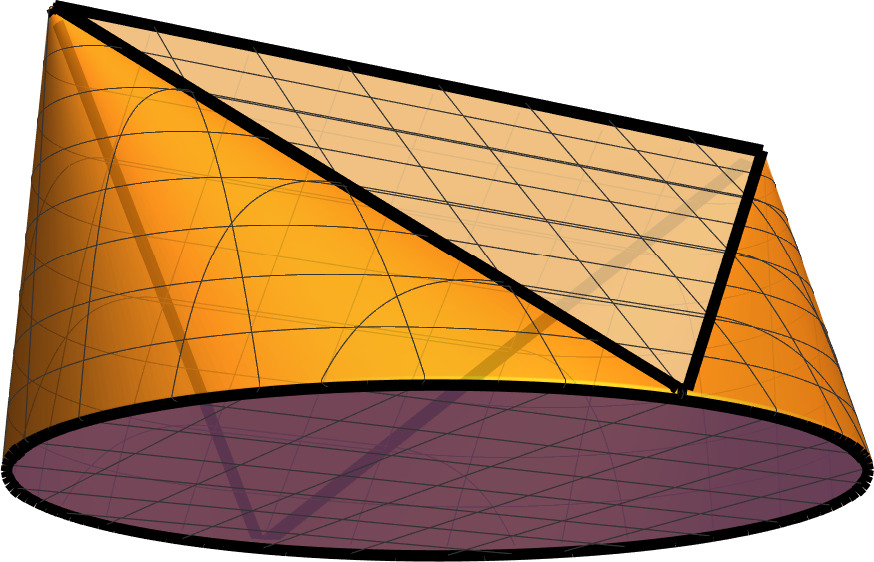}
\caption{The joint numerical ranges showing adjacency of elliptic faces and exposed points of non-elliptic ones.}\label{fig:rem6.2}
\end{figure}
To put this paper into a broader context, we quote the work \cite{Ottem-etal2015} 
by Ottem et al.\ and we sketch a proof as to why their results imply that 
$W=W(A_1,A_2,A_3)$ has generically an even number of at most ten faces of rank two
that are elliptic discs and no other faces of rank two or three, provided that 
$A_1,A_2,A_3$ are real symmetric $4\times 4$ matrices. 
\par
\begin{remark}[Generic faces are ellipses]\label{rem:Ottem}
Given real symmetric $4\times 4$ matrices 
$\widetilde{A}_0,\widetilde{A}_1,\widetilde{A}_2,\widetilde{A}_3$, the authors of 
\cite{Ottem-etal2015} define a complex projective \emph{symmetroid}
\[
\widetilde{V}=\big\{(u_0:u_1:u_2:u_3)\in\C\P^3: 
{\det}( u_0\widetilde{A}_0+u_1\widetilde{A}_1+u_2\widetilde{A}_2+u_3\widetilde{A}_3 )=0\big\}
\]
and a real projective \emph{spectrahedron}
\[
\widetilde{S}=\big\{(u_0:u_1:u_2:u_3)\in\R\P^3: 
u_0\widetilde{A}_0+u_1\widetilde{A}_1+u_2\widetilde{A}_2+u_3\widetilde{A}_3\geq 0\big\}.
\]
They assume $\dim\widetilde{S}=3$, which implies that the linear span of 
$\widetilde{A}_0,\widetilde{A}_1,\widetilde{A}_2,\widetilde{A}_3$ contains a positive 
definite matrix. Hence, a real projectivity transforms the symmetroid $\widetilde{V}$ 
to the symmetroid
\[
V:=\big\{(u_0:u_1:u_2:u_3)\in\C\P^3: 
\det(u_0\id_4 +u_1A_1 +u_2A_2 +u_3A_3) =0\big\}
\]
and the spectrahedron $\widetilde{S}$ to the spectrahedron 
\[
S:=\big\{(u_0:u_1:u_2:u_3)\in\R\P^3: 
u_0\id_4+u_1A_1+u_2A_2+u_3A_3\geq 0\big\},
\]
where $A_1,A_2,A_3$ are linearly independent real symmetric $4\times 4$ matrices
of trace zero. It follows that $S$ does not intersect the hyperplane $u_0=0$ at 
infinity and so $S$ is a compact convex subset of the affine space $\R^3\cong\{u_0=1\}$ 
that has a well-known geometry \cite{NetzerPlaumann2023}.
\par
Ottem et al.\ \cite{Ottem-etal2015} restrict their analysis to \emph{transversal symmetroids} 
$\widetilde{V}$ and \emph{transversal spectrahedra} $\widetilde{S}$, which means that 
$\widetilde{V}$ has exactly ten points of (matrix) rank two and no other points of rank smaller 
than three (singular points). The symmetroid $\widetilde{V}$ is transversal for all quadruples 
$(\widetilde{A}_0,\widetilde{A}_1,\widetilde{A}_2,\widetilde{A}_3)$ in the complement 
of an algebraic set that is strictly included in the set of all quadruples of real 
symmetric $4\times 4$ matrices. The singularities are characterized by tangent cones: 
If zero is in $\widetilde{V}$ in an affine chart of $\C\P^3$, then the 
\emph{tangent cone} of $\widetilde{V}$ at zero is the lowest degree part of 
${\det}(u_0\widetilde{A}_0+u_1\widetilde{A}_1+u_2\widetilde{A}_2+u_3\widetilde{A}_3)$
in that chart. The tangent cone at a point of rank two is thus a quadratic cone. 
We have the following equivalence:
\begin{align}\label{eq:rank-two}
& \text{(i) the point $u_0\id_4+u_1A_1+u_2A_2+u_3A_3$ has rank two,}\\
& \text{(ii) the tangent cone of $V$ at $u_0\id_4+u_1A_1+u_2A_2+u_3A_3$ is a quadratic cone.}\nonumber
\end{align}
It is shown in \cite{Ottem-etal2015}, Theorem~1.1, that an even number $\sigma$ of the 
ten points of rank two of $\widetilde{V}$ are contained in $\widetilde{S}$, provided 
that $\widetilde{V}$ is transversal.
\par
Since $W=\{x\in\R^3\colon \forall u{\in}\, S, 1+\langle x,u\rangle \geq 0 \}$ 
is the \emph{dual convex set} to $S$, see \cite[Corollary~5.3]{Plaumann-etal2021},
and since the origin $0\in\R^3$ is an interior point of the compact convex set $S$, 
it follows that every exposed face of $W$ is the intersection of the Euclidean boundary 
$\partial W$ of $W$ with the normal cone $N_S(u)$ of $S$ at some point 
$u\in S$, see \cite[Section~8]{Weis2014}. The \emph{dual convex cone} 
$\{v\in\R^3\colon\forall x{\in}\, N_S(u), \langle x,v\rangle\geq 0 \}$ to the normal 
cone $N_S(u)$ is the \emph{support cone} $T_S(u)$, which is defined as the Euclidean 
closure of the convex cone $\bigcup_{\lambda>0}\lambda(S-u)$, see 
\cite[Section~2.2]{Schneider2014}. Thus, up to a translation along $u$, the Euclidean
boundary $\partial T_S(u)$ of $T_S(u)$ is the set of points on the rays that are 
limits of rays emanating from $u$ that intersect $\partial S$ in points that converge 
to $u$. This implies that the \emph{algebraic boundary} of $T_S(u)$, which 
is the smallest complex projective algebraic set that contains $\partial T_S(u)$, is 
the tangent cone of the symmetroid $V$ at $u$, see \cite[Section~2.1.5]{Shafarevich2013}.
Thus, the assertions in \eqref{eq:rank-two} are equivalent to each of the following ones
if $u=(u_1,u_2,u_2)\tp\in S$ are the coordinates of a point in $\R^3\cong\{u_0=1\}$.
\begin{align}\label{eq:rank-two-cont}
& \text{(iii) the support cone $T_S(u)$ is an elliptic convex cone,}\\\nonumber
& \text{(iv) the normal cone $N_S(u)$ is an elliptic convex cone,}\\\nonumber
& \text{(v) the face $N_S(u)\cap\partial W$ of $W$ is an ellipse.}
\end{align}
\par
The preceding two paragraphs show that $W(A_1,A_2,A_3)$ has an even number 
$\sigma\leq 10$ of rank-two faces that are elliptic disks and no other faces of rank 
two or three, provided that $S$ is a transversal spectrahedron. This is the case for 
all $(A_1,A_2,A_3)$ in the complement of an algebraic set that is strictly included 
in the set of all triples of real symmetric $4\times 4$ matrices. In particular, it 
is true for all triples in an open and dense subset in the Euclidean topology.
\end{remark}

To clarify what we mean by obtaining $(\id_4,A_1,A_2,A_3)$ from 
$\widetilde{A}_0,\widetilde{A}_1,\widetilde{A}_2,\widetilde{A}_3$ through a 
real projectivity in Remark~\ref{rem:Ottem} above, we assume for simplicity that 
$A:=\widetilde{A}_0+\widetilde{A}_1+\widetilde{A}_2+\widetilde{A}_3$ is positive definite.
Then the projectivity $(u_0:u_1:u_2:u_3)\mapsto(u_0:u_1-u_0:u_2-u_0:u_3-u_0)$
maps $\widetilde{V}$ to 
\begin{align*}
   & 
\{(u_0:u_1:u_2:u_3)\in\C\P^3:\det(u_0A+u_1\widetilde{A}_1+u_2\widetilde{A}_2+u_3\widetilde{A}_3)=0\}\\
 = \, & 
\{(u_0:u_1:u_2:u_3)\in\C\P^3:\det(u_0\id_4+u_1A_1+u_2A_2+u_3A_3)=0\},
\end{align*}
where $A_i:=A^{-\frac{1}{2}}\widetilde{A}_i A^{-\frac{1}{2}}$, $i=1,2,3$.
Further projectivities allow us to replace $A_1,A_2,A_3$ with traceless matrices.
The matrix $A$ is indeed positive definite and the symmetroid $\widetilde{V}$ 
is transversal in the examples in \cite[Section~2]{Ottem-etal2015}. For these
examples $W(A_1,A_2,A_3)$ has an even number $\sigma\leq 10$ of rank-2 faces 
that are elliptic disks and no other faces of rank two or three.
\par\bigskip
One of the important results of \cite{Ottem-etal2015} is that every transversal spectrahedron
$S$ has an even number $\sigma\leq 10$ of points of rank two and no other points of rank
smaller than three. By Remark~\ref{rem:Ottem}, this can be interpreted by saying that
the joint numerical range $W(A_1,A_2,A_3)$ dual to $S$ has an even number $\sigma\leq 10$
of rank-2 faces that are elliptic disks and no other faces of rank two and three.
However, we found joint numerical ranges with $1,3,5$ elliptic faces --- this is possible
because the dual spectrahedra are not transversal, 
as some elliptical faces are of rank three. Similarly, the spectrahedra dual to the 
Examples~\hyperref[ex-1]{E1}--\hyperref[ex-14]{E14} in Section~\ref{sec:examples} 
above cannot be transversal.
\par

Consider the joint numerical range with five elliptic faces
(shown in Figure~\ref{fig:elliptic}~(a)),
obtained for the following three matrices:
\begin{equation*}\begin{bmatrix}
 0 & 0 & 1 & 0 \\
 0 & 0 & 0 & 0 \\
 1 & 0 & 0 & 0 \\
 0 & 0 & 0 & 0 \end{bmatrix},\begin{bmatrix}
 0 & 0 & 0 & 0 \\
 0 & 0 & 0 & 1 \\
 0 & 0 & 0 & 0 \\
 0 & 1 & 0 & 0 \end{bmatrix},\begin{bmatrix}
 1 & 0 & 0 & 0 \\
 0 & 0 & 1 & 0 \\
 0 & 1 & 0 & 0 \\
 0 & 0 & 0 & 1 \end{bmatrix}
 \end{equation*}
\medskip

In this case the relevant spectrahedron is not transversal,
since the top elliptic face is of rank 3, and can be regarded
as an affine linear image of the numerical range $W(B_1, B_2)$
of the following matrices: 
\[
B_1=\begin{bmatrix}0&0&1\\0&0&0\\1&0&0\end{bmatrix},
\quad B_2=\begin{bmatrix}0&1&0\\1&0&0\\0&0&0\end{bmatrix}.
\]
This illustrates Remark \ref{rem:bordered}: these are bordered matrices, 
resulting in elliptic numerical range.

Without the exact knowledge of the projectors onto subspaces associated with the
elliptic faces, it is not possible to determine whether or not they intersect:
simple reasoning based on dimension counting like in Lemma~\ref{lem:int2non-elliptic}
are indecisive. Compare the following two examples with the same structure of subspaces
dimensions: each of the matrices in both examples has two double eigenvalues
of $-1$ and $+1$, but there are significant differences in geometry of the
joint numerical ranges.

\begin{figure}[htb]\centering
(a) \includegraphics[width=3.2cm]{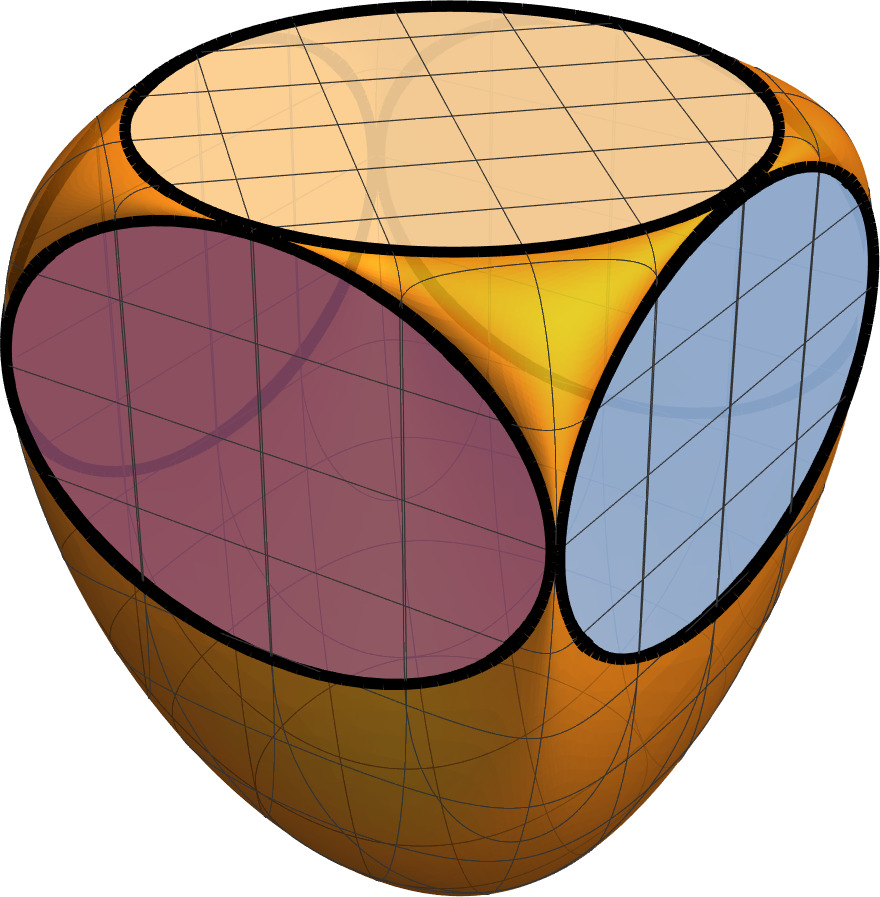}\qquad
(b) \includegraphics[width=3.2cm]{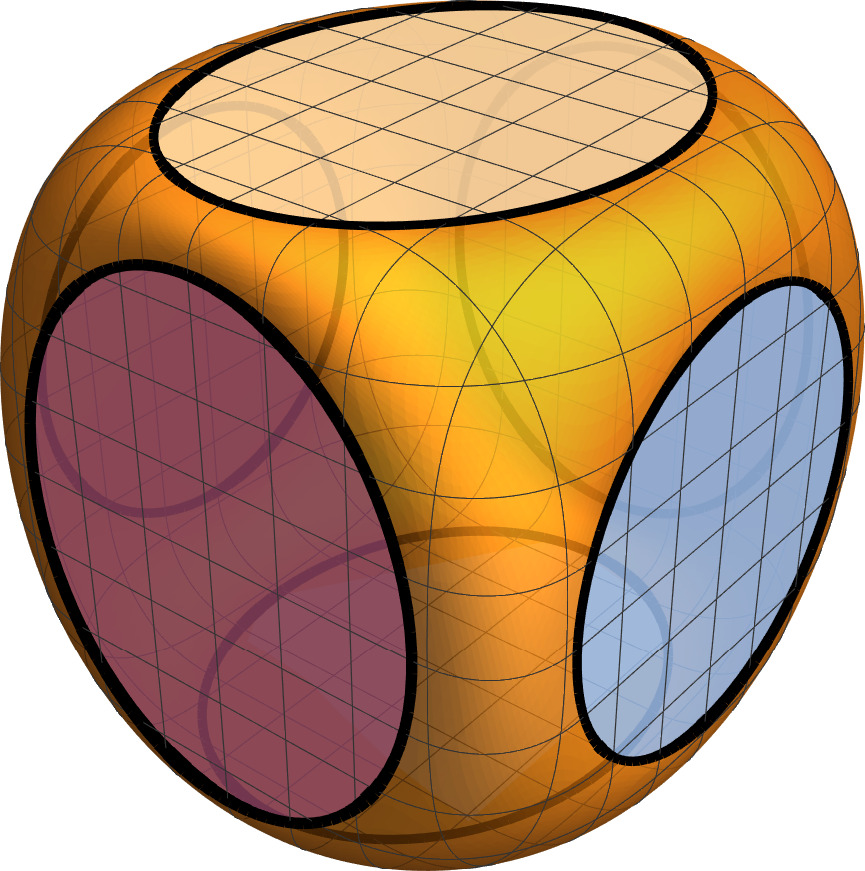}\qquad
(c) \includegraphics[width=3.2cm]{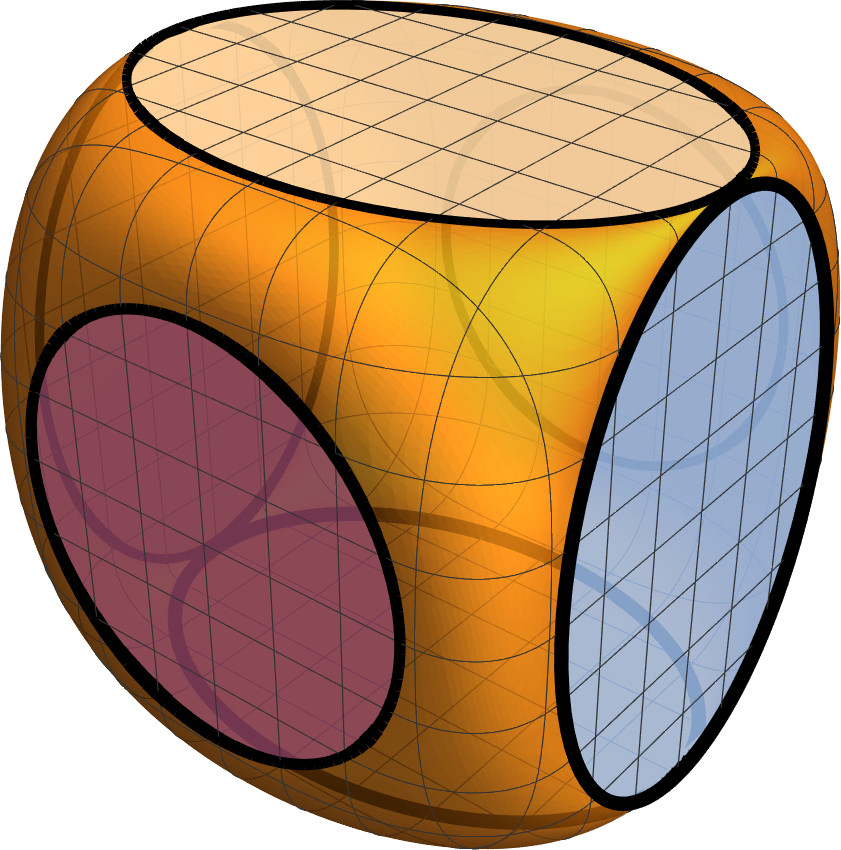}
\caption{Examples of joint numerical ranges for triples of $4\times 4$ hermitian matrices with various configurations
of elliptic faces: (a) five faces, (b) six non-adjacent, (c) six
with two disjoint from the remaining four.}\label{fig:elliptic}
\end{figure}

One of joint numerical ranges forms a `dice' and contains six elliptic
faces, which are pairwise disjoint (see Figure~\ref{fig:elliptic}~(b)), and is obtained for the following matrices:
\begin{equation*}\frac{1}{\sqrt{2}}
\begin{bmatrix}
 1 & 1 & 0 & 0 \\
 1 & -1 & 0 & 0 \\
 0 & 0 & -1 & 1 \\
 0 & 0 & 1 & 1 \end{bmatrix},
 \begin{bmatrix}
 0 & -1 & 0 & 0 \\
 -1 & 0 & 0 & 0 \\
 0 & 0 & 0 & 1 \\
 0 & 0 & 1 & 0 \end{bmatrix},\begin{bmatrix}
 0 & 0 & 0 & -1 \\
 0 & 0 & \ii & 0 \\
 0 & -\ii & 0 & 0 \\
 -1 & 0 & 0 & 0 \end{bmatrix}.
 \end{equation*}

A joint numerical range with six elliptic faces, four of them intersecting
in a ring pattern, and two disjoint (see Figure~\ref{fig:elliptic}~(c)), corresponds to the following triple
of symmetric matrices:
\begin{equation*}\hspace{-10mm}\frac{1}{\sqrt{2}}\begin{bmatrix}
 1 & 1 & 0 & 0 \\
 1 & -1 & 0 & 0 \\
 0 & 0 & -1 & 1 \\
 0 & 0 & 1 & 1 \end{bmatrix},\begin{bmatrix}
 0 & -1 & 0 & 0 \\
 -1 & 0 & 0 & 0 \\
 0 & 0 & 0 & 1 \\
 0 & 0 & 1 & 0 \end{bmatrix},\begin{bmatrix}
 0 & 0 & 0 & -1 \\
 0 & 0 & 1 & 0 \\
 0 & 1 & 0 & 0 \\
 -1 & 0 & 0 & 0 \end{bmatrix} .
 \end{equation*}
%
%
\section{Separable numerical range}
\label{sec:restricted-JNRs}

Some applications in quantum physics rely of the notion of \emph{restricted numerical range},
in which the set of states $\cD_A$ is replaced with a specific subset \cite{GPM+10}.
In the case of systems composed of two parts -- e.g., two qubits -- this subset might
correspond to states which are in some sense classical, and the numerical range restricted
to such states highlights the difference between classical and quantum behaviors of
quantum systems.

Let us start with the mathematical description of the physical scenario. If the physical 
operations and observable quantities of the first system are described by the set $\M_d$ 
of matrices of size d (forming a *-algebra on $\mathbb{C}^d$), and the operations and 
observables on the second subsystem by $\M_{d'}$ (analogously, a *-algebra on 
$\mathbb{C}^{d'}$), the total physical bipartite (two-party) system corresponds to the 
tensor product of algebras.
As in Section~\ref{sec:recap} above, we denote the set of density matrices of size $d$ by 
$\cD_d$. With this notation, the density matrices of the entire system can be denoted 
by  $\cD_{d\times d'}$: they are positive semidefinite, unit trace matrices of order 
$d\cdot d'$.
Here, we mostly concentrate on two-qubit systems, as this matches the topic of this paper: 
two qubits correspond to the state space $\cD_4$ needed in the definition of joint numerical 
ranges for matrices of size four. 

One of the most interesting features of such a description is that some states give 
rise to correlations that can not be explained by classical probability theory. In some sense, the "classical" nature of some states is captured by the fact that they can be decomposed into a convex sum of tensor products of one-subsystem states -- such states are called \emph{separable} \cite{Horodeccy1996}:
\begin{equation}\label{eq:sepstates}
\cD^{\mathrm{sep}}_{d\times d'} = \operatorname{conv}\{\  \rho^A \otimes \rho^B : \rho^A \in \cD_d ,\rho^B\in\cD_{d'}\}.
\end{equation}

Here we study the \emph{separable states of two qubits} $\cD^{\mathrm{sep}}_{2\times 2}$, as this is a proper subset of $\cD_4$, projections of which are the main topic of this article. By analogy to the definition of joint numerical range (see \eqref{eq:JNR-A}), we define the \emph{separable numerical range} \cite{GPM+10}
\begin{equation}\label{eq:sepnrange}
W^{\mathrm{sep}}(A_1,\ldots, A_k)=w(\cD^{\mathrm{sep}}_{d\times d'})=\{(\tr \rho A_1, \ldots, \tr \rho A_k) : \rho\in\cD^{\mathrm{sep}}_{d\times d'}\}.
\end{equation}
The object defined in such a way (as well as the related product numerical range, where only
rank one product states are taken into account, leading to nonconvexity \cite{PGM+11}),
are important tools in quantum information theory, allowing for the study of
quantum entanglement \cite{SCSZ21}.
In practice, the tensor product, which is a central point to the study of entanglement,
is typically realized as the Kronecker product. Temporarily employing the physical 
notation of $\ket{0}=(1,0)\tp\in\mathbb{C}^2$ and
$\ket{1}=(0,1)\tp\in\mathbb{C}^2$, the Kronecker products
$\ket{ij}=\ket{i}\otimes\ket{j}$ of the two vectors take the following form:
\begin{equation*}
\ket{00}=\begin{pmatrix}1\\0\\0\\0\end{pmatrix},\quad \ket{01}=\begin{pmatrix}0\\1\\0\\0\end{pmatrix},\quad
\ket{10}=\begin{pmatrix}0\\0\\1\\0\end{pmatrix},
\text{ and~}\ket{11}=\begin{pmatrix}0\\0\\0\\1\end{pmatrix}.
\end{equation*}
Similarly, the tensor product appearing in \eqref{eq:sepstates} is here interpreted as
\begin{equation*}
    \begin{pmatrix}A_{11} & A_{12} \\ A_{21} & A_{22}\end{pmatrix}\otimes
    \overbrace{\begin{pmatrix}B_{11} & B_{12} \\ B_{21} & B_{22}\end{pmatrix}}^{B} = \begin{pmatrix}A_{11} B & A_{12} B \\ A_{21} B & A_{22}B\end{pmatrix}.
\end{equation*}
Physical motivation here is to characterize the difference between entangled
and separable states through the associated numerical ranges
-- with this in mind, we analyze the question:
how the shape of the boundary $\partial W$ of the standard joint numerical range
of three hermitian matrices of order four, investigated in previous sections,
influences the structure of the corresponding separable numerical range $W^{\mathrm{sep}}$.
     
The particular case of two-qubit system is 
not difficult to probe numerically through the use of semidefinite optimization due to the 
positive partial transpose (PPT) criterion \cite{Peres1996,Horodeccy1996}:
\begin{theorem}[{Peres-Horodeccy}]
\label{thm:pptcrit}
For $(d,d')=(2,2)$ and $(2,3)$, the states belonging to $\cD^{\mathrm{sep}}_{d,d'}$
are characterised by semidefinite positivity of their partial transpose:
\begin{equation*}
\rho\in\cD^{\mathrm{sep}}_{d\times d'} \Leftrightarrow \rho \in \cD_{d\times d'} \text{~and~}\rho^\Gamma \in \cD_{d\times d'},
\end{equation*}
where $\cdot^\Gamma$ denotes the partial transpose. For the matrix $X\in \M_{2\times d}$
of square blocks $\{Y,Z,U,V\}$ of size $d$
\begin{equation*}
X=\begin{pmatrix}Y & Z\\U&V\end{pmatrix},
\end{equation*}
the partial transpose is the transpose of each of the blocks:
\begin{equation*}
X^\Gamma=\begin{pmatrix}Y\tp & Z\tp\\U\tp&V\tp\end{pmatrix}
\end{equation*}
\end{theorem}

This condition can be used together with semidefinite optimization in order to determine
exposed faces of $W^{\mathrm{sep}}$, allowing for numerical study and visualisations of 
the set. In what follows, we consider two qubits, where $(d,d')=(2,2)$ and 
$A_1, A_2, A_3$ are hermitian matrices of order four.
\par
\begin{lemma}
Let $F$ be an exposed face of $W^{\mathrm{sep}}(A_1, A_2, A_3)$ minimizing the linear 
functional $W^{\mathrm{sep}}\ni p \mapsto \sum_{i=1}^3 s_i p_i$. The preimage of $F$ 
under the measure map given by \eqref{eq:measuremap} is characterized as the solutions 
to the following semidefinite optimization problem:
\begin{equation}
\begin{aligned}
\min_{\rho} \tr  \rho \left(\sum_{i=1}^k A_i s_i\right)\\
\text{subject to } \rho \geq 0,\ \rho^\Gamma 
\geq 0,\ \tr  \rho =1.
\label{eq:sepopt}
\end{aligned}
\end{equation}
\end{lemma}
\begin{proof}
If a point $p\in W^{\mathrm{sep}}$, there must exist
a hermitian matrix $\rho$ of order four 
such that $w(\rho)=p$, and by Theorem \ref{thm:pptcrit} it must fulfill $\tr  \rho=1$, $\rho{\ge}0$ and $\rho^\Gamma{\ge}0$.
If $p \in F$, for any other $\sigma\in\cD^{\mathrm{sep}}_{2\times 2}$, $\sum_{i=1}^3 s_i \tr (\rho A_i)\le \sum_{i=1}^3 s_i \tr (\sigma A_i)$,
and therefore $\rho$ is a solution to the semidefinite optimization problem \eqref{eq:sepopt}. 
\end{proof}
\begin{figure}[t]
\centering
\begin{subfigure}[t]{.45\textwidth}
\centering
\includegraphics[width=.78\textwidth]{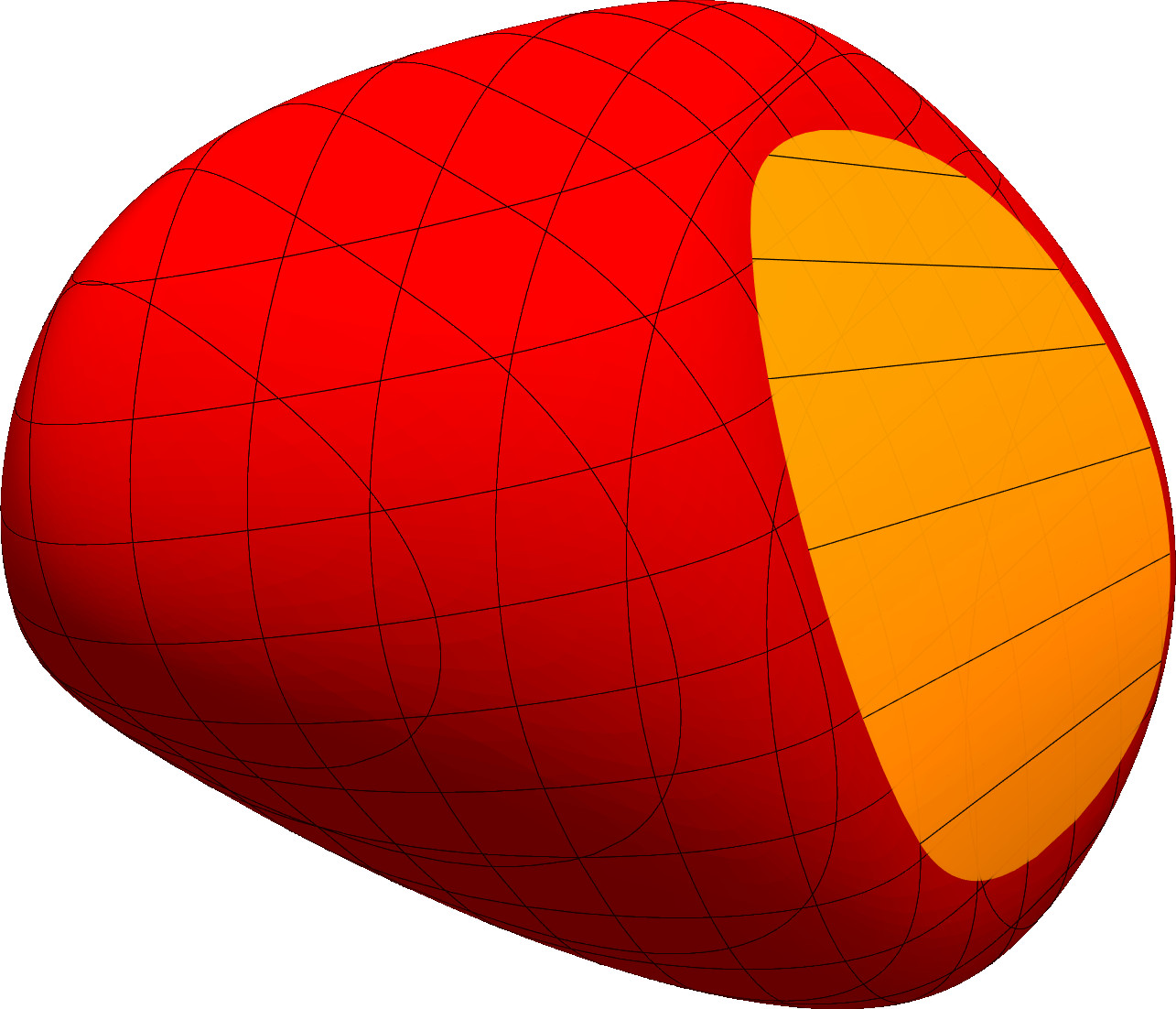}
\caption{}\label{fig:snrs-a}
\end{subfigure}
\hfill
\begin{subfigure}[t]{.45\textwidth}
\centering
\includegraphics[width=.78\textwidth]{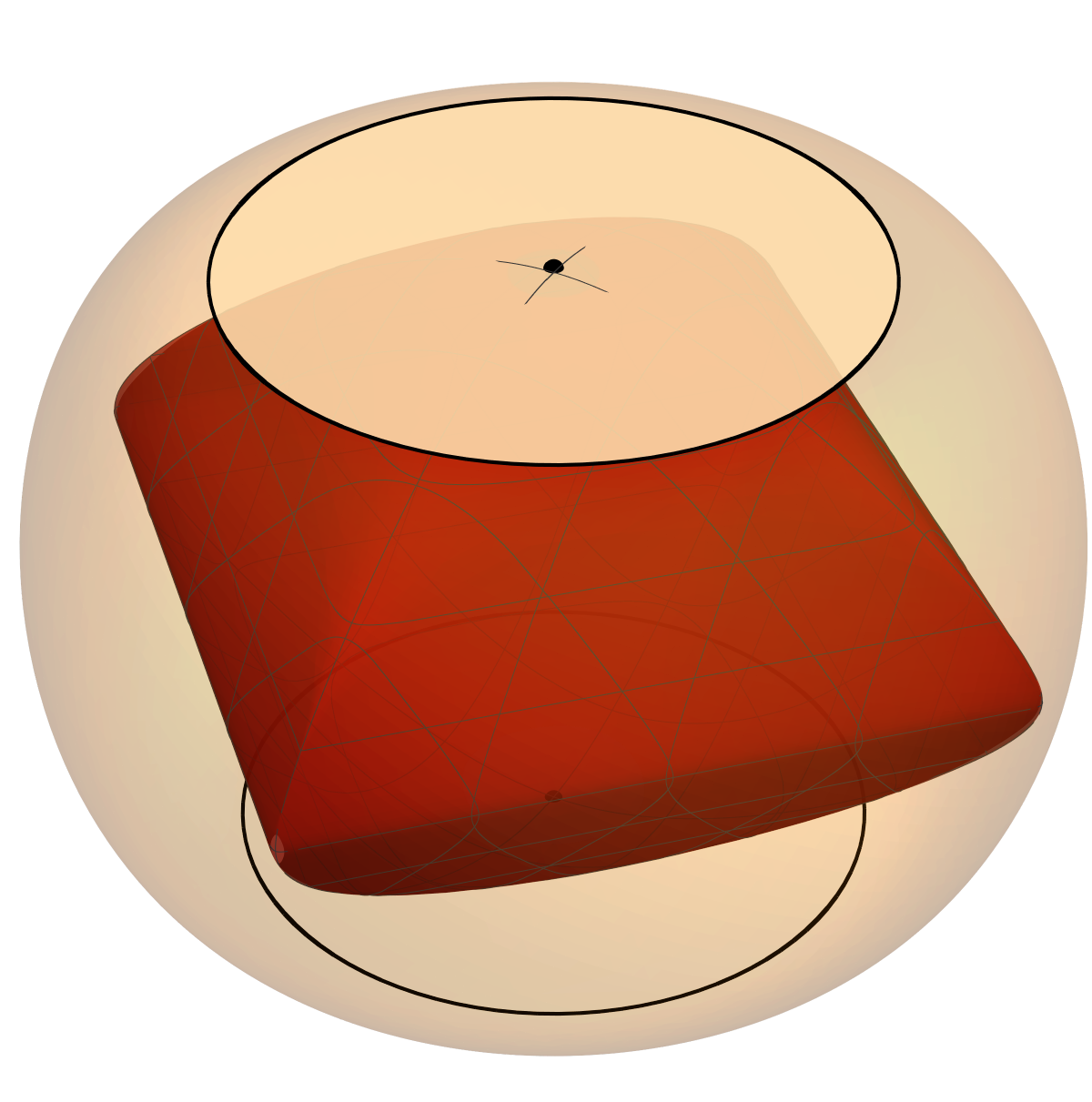}
\caption{}\label{fig:snrs-b}
\end{subfigure}
\caption{(a) Separable numerical range (red) for randomly sampled matrices 
\eqref{eq:randommat}; according to Remark \ref{rem:numerics}, the boundary contains 
a ruled surface (yellow region) some of whose lines are depicted.
(b) Joint and separable numerical ranges (transparent yellow and solid red, respectively)
for the matrices \eqref{eq:sepmat} illustrate Lemma~\ref{lem:facesnr}: 
the two circular faces of the joint numerical range (at the top and bottom) 
intersect the separable numerical range (black dots).}\label{fig:snrs}
\end{figure}

 For visualisation purposes it is sufficient to sample the functional parameters
$(n_1, n_2, n_3)$ from the unit sphere $S_2$, and for each numerically determine
an optimizer $\rho$ to the above minimization problem. If the optimization converges,
the point $w(\rho)$ lies at the boundary of the separable numerical range, within prescribed
optimization accuracy (which is well controlled for semidefinite problems \cite{invnonlin}).
Then, the convex hull of the sampled boundary points approximates
$W^{\mathrm{sep}}(A_1, A_2, A_3)$. 

To the best our knowledge, this approximation procedure is the only way of studying the
separable numerical ranges, if the matrices do not possess any additional structure.
The numerically probed geometry of three-dimensional $W^{\mathrm{sep}}$ always seems
to contain singular features\footnote{This is in contrast with two-dimensional
separable numerical ranges, which for generic matrices do not contain any boundary segments.
Note that due to the numerical approximation procedure, the decision
whether a given region of the boundary is flat is somewhat arbitrary;
here we take the ruled surfaces to be the regions with large segments
between the vertices of the convex hull, which persist despite subsampling.}. 

\begin{remark} [Numerical empirical finding]\label{rem:numerics}
For hermitian matrices $A_1, A_2, A_3$ sampled i.i.d.\ from the
Gaussian Unitary Ensemble of matrices of size $4$, the boundary of the separable
numerical range $W^{\mathrm{sep}}(A_1, A_2, A_3)$ always contains singular elements:
cusps, flat faces, or ruled surfaces. Right now it is only a numerical observation,
as we are not aware of any techniques to prove that singular features do exist in this case.
\end{remark}

We attribute this observation to possible construction of $W^{\mathrm{sep}}$
as a convex hull of the analogous product numerical range \cite{PGM+11},
\begin{equation*}
W^{\mathrm{sep}}=\operatorname{conv} W^{\otimes},
\text{ with } W^\otimes= w_{A_1, A_2, A_3}(
\{\rho\in\cD^{\mathrm{sep}} : \operatorname{rk} \rho = 1\}).
\end{equation*}
The set $W^{\otimes}$ appears to be nonconvex under the assumptions of the above Remark,
and ruled surfaces form in $W^{\mathrm{sep}}$ as a result of this.
 An example of the above observation is provided in Figure \ref{fig:snrs}a.
There, the ruled surface of the separable numerical range is highlighted;
the matrices are taken from the Gaussian Unitary Ensemble and approximated by
rational coefficients for reproducibility:
\begin{align}\label{eq:randommat}
&\begin{bmatrix}
 -3 & -9-6 \ii & -15-\ii & -2+13 \ii \\
 * & 6 & -2+5 \ii & -2+10 \ii \\
 * & * & -10 & -6-7 \ii \\
 * & * & * & 3 
\end{bmatrix},
\quad
\begin{bmatrix}
 11 & 9+2 \ii & 2-5 \ii & -1-2 \ii \\
 * & -9 & -4-3 \ii & 1-4 \ii \\
 * & * & -16 & 10-12 \ii \\
 * & * & * & -2 
\end{bmatrix},\\
\nonumber
&\begin{bmatrix}
 15 & 10-3 \ii & 5+5 \ii & -\ii \\
 * & -6 & -16-2 \ii & 15 \\
 * & * & 3 & -9-7 \ii \\
 * & * & * & -13 
\end{bmatrix}
\end{align}
Numerically, the ruled surfaces do not seem to be caused by the rational structure of 
coefficients. Their existence is a generic behavior. Practically, the matrices were 
generated in Mathematica 14.1 with the following code.
{\small\begin{align*}
&\texttt{SeedRandom[2137];}\\    
& \texttt{10 Round[ RandomVariate[GaussianUnitaryMatrixDistribution[4], 3], 1/10]}
\end{align*}}
\par
The separable numerical range must lie within the joint numerical range, as separable states 
form a subset of all states. The two ranges can be tangent to each other, and this is known 
to happen if the joint numerical range has a face that is not a point.
\par
\begin{lemma}\label{lem:facesnr}
Every rank-$2$ and rank-$3$ face of the joint numerical range $W(A_1, A_2, A_3)$ 
intersects the separable numerical range $W^{\mathrm{sep}}(A_1, A_2, A_3)$.
In particular, every face of $W(A_1, A_2, A_3)$ that is larger than a singleton 
intersects $W^{\mathrm{sep}}(A_1, A_2, A_3)$.
\end{lemma}
\begin{proof}
Every rank-$2$ or rank-$3$ face is an image (under $w$) of a subset $S$ of $\cD_4$ 
isomorphic with $\cD_2$ or $\cD_3$, respectively, by Theorem~\ref{thm:PF-iso}.
In particular, his is true for every face that is larger than a singleton.
\par
This set $S$ is the convex hull of the rank-one projectors onto vectors belonging 
to a 2- or 3-dimensional subspace of $\mathbb{C}^4$. It is known that any 
2-dimensional subspace of $\mathbb{C}^4$ contains a pure tensor $v_A\otimes v_B$ 
(see \cite[Theorem~2]{Sanpera-etal1998}). So the projector $\Pi$ onto $v_A\otimes v_B$ 
is contained in $S\cap\cD^{\mathrm{sep}}_{2\times 2}$ and $w(v_A\otimes v_B)$ is 
contained in $W^{\mathrm{sep}}(A_1,A_2,A_3)$.
\end{proof}
Lemma~\ref{lem:facesnr} is illustrated in Figure \ref{fig:snrs}b
for the following matrices.
\begin{align}\label{eq:sepmat}
\begin{bmatrix}
 0 & 1 & 0 & 1 \\
 * & 0 & 1 & 0 \\
 * & * & 0 & 0 \\
 * & * & * & 0
\end{bmatrix},
\quad
\begin{bmatrix}
 0 & -\ii & 0 & \ii \\
 * & 0 & -\ii & 0 \\
 * & * & 0 & 0 \\
 * & * & * & 0 
\end{bmatrix},
\quad
\diag(-1,1,1,-1)
\end{align}
%
%
\subsection*{Acknowledgments}
SW thanks Cynthia Vinzant for helpful discussions on convex duality and on the results of 
\cite{Ottem-etal2015}. KS gratefully acknowledges funding from the projects DeQHOST APVV-22-0570, QUAS
VEGA 2/0164/25, Postdokgrant APD0161, and Stefan Schwarz programme.
SW was supported by the European Union under the project ROBOPROX 
(reg.\ no.\ CZ.02.01.01/00/22\_008/0004590),
while K{\.Z} acknowledges support 
under ERC Advanced Grant Tatypic, Project No. 101142236.
%
%

\end{document}